\documentclass[letterpaper,11pt]{article}
\RequirePackage{amsmath, amsthm, amssymb, mathtools}
\RequirePackage{comment}
\RequirePackage{graphicx}
\RequirePackage{subcaption} 
\RequirePackage{fancyvrb}
\RequirePackage[dvipsnames,svgnames,table]{xcolor}
    \definecolor{linkcolor}{RGB}{0,0,128}
\RequirePackage[pdfpagelabels,plainpages=false,hypertexnames=true,colorlinks=true,linkcolor=linkcolor,anchorcolor=linkcolor,citecolor=linkcolor,filecolor=linkcolor,menucolor=linkcolor,runcolor=linkcolor,urlcolor=linkcolor,pdfborder={0 0 0}]{hyperref}
\RequirePackage[scale=0.95]{sourcecodepro}
\RequirePackage{mathpazo}

\RequirePackage{dsfont}
\RequirePackage{parskip}
\RequirePackage{tikz}
    \usetikzlibrary{positioning}
    \usetikzlibrary{decorations.pathreplacing}
\RequirePackage[backend=biber,backref=true,style=numeric,giveninits=true,natbib=true,maxbibnames=12]{biblatex}
\RequirePackage{microtype}
\RequirePackage{forest}

\usepackage{float}
\usepackage{romanbar}

\usepackage{orcidlink}

\usepackage{comment}
\usepackage{todonotes}
\usepackage{pgfplots}
\pgfplotsset{compat=1.18} 

\usepackage{algorithm}

\theoremstyle{plain}
\newtheorem{prototheorem}{theorem}
\newtheorem{theorem}[prototheorem]{Theorem}

\theoremstyle{plain}
\newtheorem{prototheorem2}{theorem}
\newtheorem{problem}[prototheorem2]{Problem}

\theoremstyle{plain}
\newtheorem{prototheorem3}{theorem}
\newtheorem{lemma}[prototheorem3]{Lemma}

\theoremstyle{remark}
\newtheorem{remark}{Remark}

\theoremstyle{definition}

\newcommand{\eps}{\varepsilon}

\newcommand{\J}{\mathrm{Couplings}}
\newcommand{\law}{\operatorname{Law}}

\newcommand{\TV}{\mathsf{TV}}

\newcommand{\rn}[1]{\Romanbar{#1}}

\newcommand{\tmix}{\tau_{\operatorname{mix}}}

\newcommand{\ind}{\mathbf{1}}
\newcommand{\diam}{\operatorname{diam}}

\newcommand{\AR}{\mathrm{a/r}}
\newcommand{\A}{\mathrm{a}}
\newcommand{\R}{\mathrm{r}}

\newcommand{\Mult}{\mathrm{Multinoulli}}

\newcommand{\Unif}{\mathrm{Unif}}
\newcommand{\NUTS}{\mathrm{NUTS}}

\graphicspath{{Figures/}}

\title{Mixing of the No-U-Turn Sampler and the \\ Geometry of Gaussian Concentration}
\author{Nawaf Bou-Rabee\thanks{Department of Mathematical Sciences, Rutgers University, \href{mailto:nawaf.bourabee@rutgers.edu}  {\texttt{nawaf.bourabee@rutgers.edu}} \orcidlink{0000-0001-9280-9808}}
\and
Stefan Oberd\"orster\thanks{Institute for Applied Mathematics, University of Bonn, 
\href{mailto:oberdoerster@uni-bonn.de}{\texttt{oberdoerster@uni-bonn.de}} \orcidlink{0009-0004-0734-0101}}
}

\date{October 2024}

\begin{document}

\maketitle

\begin{abstract}
We prove that the mixing time of the No-U-Turn Sampler (NUTS), when initialized in the concentration region of the canonical Gaussian measure, scales as $d^{1/4}$, up to logarithmic factors, where $d$ is the dimension.  This scaling is expected to be sharp.  The result is based on a coupling argument that leverages the geometric structure of the target distribution.  Specifically, concentration of measure results in a striking uniformity in NUTS' locally adapted transitions, which holds with high probability.  This  uniformity is formalized by interpreting NUTS as an accept/reject Markov chain, where the mixing properties for the more uniform accept chain are analytically tractable.  Additionally, our analysis uncovers a previously unnoticed issue with the path length adaptation procedure of NUTS, specifically related to looping behavior, which we address in detail.
\end{abstract}

\section{Introduction}

For the canonical Gaussian measure $\gamma$ on $\mathbb R^d$,  the solution to the isoperimetric problem implies that for any set $A$ with $\gamma(A)\geq1/2$, the $r$-neighborhood of $A$, denoted by $A_r$, satisfies $\gamma(A_r)\geq1-\exp(-r^2/2)$.
Equivalently, all sufficiently regular functions are almost constant, or very close to their means, on almost all of space.
For the norm, this translates into
\begin{equation}\label{eq:Gauss_conc}
    \gamma\bigr(\bigr||x|-\sqrt d\bigr|>r\bigr)\ \leq\ 2\,\exp(-r^2/2)\;,
\end{equation}
which shows that $\gamma$  concentrates within a thin spherical shell around the $(d-1)$-sphere of radius $\sqrt{d}$, i.e., $\sqrt d\,\mathcal S^{d-1}$  \cite{Talagrand_NewLook,TalagrandLedoux_ProbabilityIn}. This paper studies the consequences of this concentration of measure phenomenon on the mixing time of the No-U-Turn sampler (NUTS).

NUTS is a Markov chain Monte Carlo (MCMC) method designed to approximately sample from an absolutely continuous target probability measure with  a continuously differentiable density  \cite{HoGe2014,betancourt2017conceptual, carpenter2016stan}.  It builds upon Hamiltonian Monte Carlo (HMC)  \cite{DuKePeRo1987,Ne2011}, which has long been known to be sensitive to path length tuning, as noted by Mackenzie in the late 1980s \cite{Ma1989}.  A quarter-century later, Hoffman and Gelman introduced NUTS as a solution to this path length tuning problem.  
Owing to this self-tuning capability, NUTS has become the default sampler in many probabilistic programming environments \cite{carpenter2016stan,salvatier2016probabilistic,nimble-article:2017,ge2018t,phan2019composable}.

In each transition step, NUTS begins by drawing a new velocity from the Gaussian measure $\gamma$.  From this velocity and the current state of the chain, NUTS generates a trajectory by repeatedly applying the leapfrog integrator, producing a sequence of states known as the leapfrog orbit. The orbit expands either forward or backward in time with equal probability, and its length doubles with each expansion (see Figure~\ref{fig:NUTS_transition_step}).  This process, referred to as \emph{orbit selection}, continues until a 'U-turn' is detected, which aims to signal that the leapfrog orbit has likely reached the limit of efficient exploration.  

Once an orbit is selected, the next state of the NUTS chain is chosen through a process called \emph{index selection}, where a state from the leapfrog orbit is selected with probability proportional to its Boltzmann weight.  Due to discretization errors in the leapfrog integrator, these Boltzmann weights are not uniform.  It is important to note that both the orbit and index selections depend on the current state of the chain.

This intricate architecture presents significant challenges for its theoretical analysis and understanding.  Even proving its reversibility remains an ongoing area of research \cite{andrieu2020general,durmus2023convergence}, with a concise proof only recently established \cite{BouRabeeCarpenterMarsden2024}. The key insight from \cite{BouRabeeCarpenterMarsden2024} is to interpret NUTS as an auxiliary variable method, where the auxiliary variables consist of: the velocity,  the orbit,  and the index, as detailed in Section~\ref{sec:NUTS}.  In this precise sense, NUTS builds on HMC by introducing additional auxiliary variables that enable local adaptation of the path length.

Fundamental questions regarding the mixing time of NUTS have remained unresolved, and  quantitative guarantees for its convergence have yet to be established.  Although existing methods such as couplings \cite{BouRabeeOberdoerster2024,BouRabeeEberle2023,BouRabeeSchuh2023,BoEbZi2020,Eberle14}, conductance \cite{Chen_Minimax,chen2023,apers2022,chewi2021optimal,LST2020}, and functional inequalities \cite{EberleLoerler24,CaoLuWang23,monmarche2024entropic,camrud2023second} have proven successful in the analysis of HMC, these methods have not been extended to account for the additional intricacies of NUTS.

In this paper, we focus on the canonical Gaussian measure $\gamma$ as a central example for a systematic analysis of the mixing time of NUTS.
We rigorously establish novel results on NUTS initialized in the region of concentration of $\gamma$, providing deeper insights into its behavior.  First, we prove a striking uniformity in NUTS' path length adaptation within this concentration region, allowing for a reduction to state-independent path length selection.  Second, we reinterpret NUTS' index selection as a Metropolis-like filter, linking NUTS to the literature on Metropolis-adjusted HMC \cite{Chen_Minimax,BouRabeeOberdoerster2024}.  Finally, we extend a recently developed coupling framework for analyzing the mixing times of Metropolis-adjusted Markov chains \cite{BouRabeeOberdoerster2024} to include NUTS, taking into account the geometric structure of the target distribution $\gamma$.

Building on these insights, our main result provides the first quantitative mixing time guarantee for NUTS when initialized near the concentration region of the canonical Gaussian measure $\gamma$.  Importantly, we require no assumptions on the path length, as NUTS adapts it locally.  This result demonstrates the effectiveness of the U-turn-based approach in the isotropic geometry of $\gamma$ and contributes to a deeper understanding of the behavior of NUTS in high-dimensional sampling problems. Additionally, we highlight a previously unnoticed issue in NUTS' path length adaptation, which was uncovered during the analysis.

\begin{figure}
\centering
\includegraphics[width=0.3\textwidth]{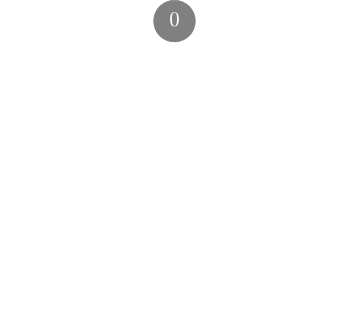}
\hspace{10pt}
\includegraphics[width=0.3\textwidth]{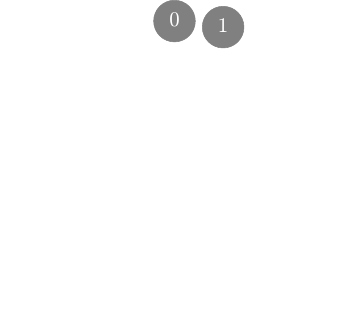}
\hspace{10pt}
\includegraphics[width=0.3\textwidth]{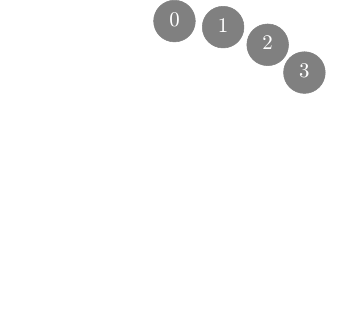} \\
\vspace{-60pt}
\includegraphics[width=0.3\textwidth] {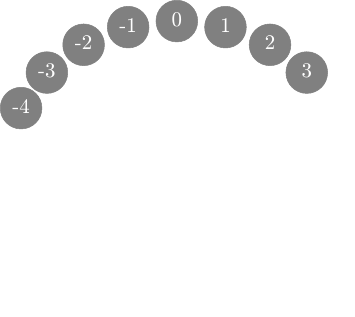} \hspace{10pt}
\includegraphics[width=0.3\textwidth]{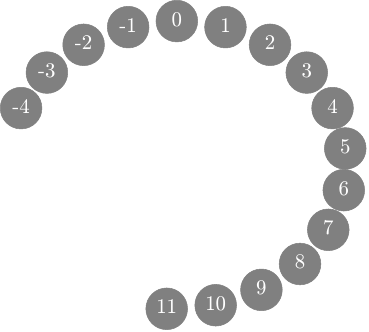}
\hspace{10pt}
\includegraphics[width=0.3\textwidth]{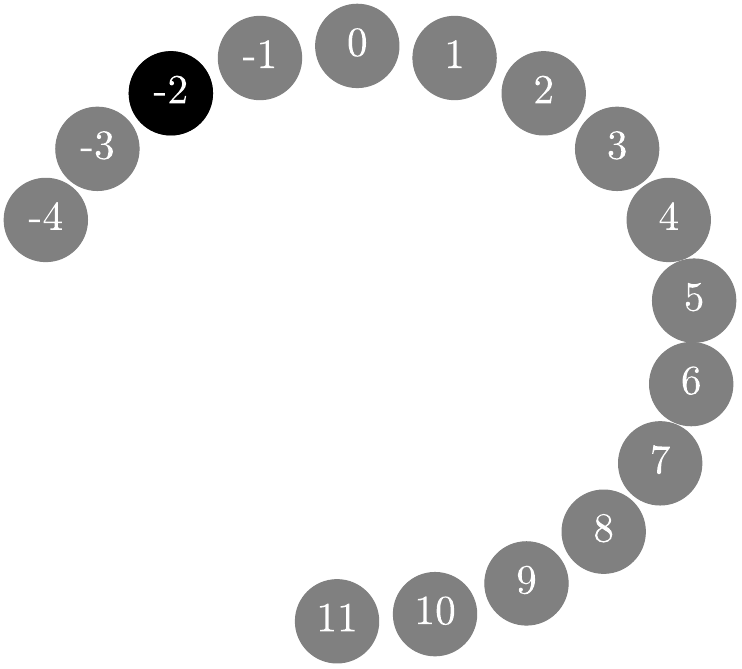} 
\caption{\textit{
In each transition step, NUTS iteratively builds a leapfrog orbit starting from the initial state, labeled as $0$. At each iteration, the orbit is doubled either forward or backward in time with equal probability, continuing until a U-turn is detected. Once a U-turn occurs, the next state of the NUTS chain is stochastically selected from the final orbit with sampling probabilities determined by the energies of the corresponding leapfrog iterates, as described in Algorithm~\ref{algo:NUTS}.
}}
\label{fig:NUTS_transition_step}
\end{figure}

\subsection{The No-U-Turn Sampler}\label{sec:NUTS}

Let $\mu(dx)\propto\exp(-U(x))\,dx$ be a given target probability measure on $\mathbb R^d$ where $U: \mathbb{R}^d \to \mathbb{R}$ is continuously differentiable.  Define $\phi_t:\mathbb{R}^{2d}\to\mathbb{R}^{2d}$ as the flow of Hamilton's equations for the Hamiltonian function
\[ H(x,v)\ =\ U(x)\ +\ |v|^2/2 \]
where $x \in \mathbb{R}^d$ and $v \in \mathbb{R}^d$ represent position and velocity, respectively.  Fix a step size $h>0$.  Let $\Phi_h: \mathbb{R}^{2d} \to \mathbb{R}^{2d}$ denote one leapfrog step of size $h$  for this flow.  For any $L \in \mathbb{Z}$, define $\Phi_{h}^L:  \mathbb{R}^{2d} \to \mathbb{R}^{2d}$ recursively by: $\Phi^0_{h}$ is the identity map, and $\Phi^{L + 1}_{h} = \Phi_{h} \circ \Phi^L_{h}$.

We now describe the transition kernel of NUTS in detail.  Given the current state $x\in\mathbb R^d$, NUTS first draws an initial velocity $v\sim\gamma$.  Next, it performs an orbit selection and then an index selection, which together determine the next state of the chain.  We begin by explaining the no-U-turn rule.

\paragraph{U-turn property.}   Define an \emph{index orbit} as a set of consecutive integers $I\subset\mathbb Z$ of length $|I|\in2^{\mathbb N} =\{2^n\}_{n\in\mathbb N}$.  Let  $\mathfrak I_0(k)$, for $k\in\mathbb N$, denote the collection of index orbits of fixed length $2^k$ that include the integer $0$.  Given initial conditions $(x,v)\in\mathbb R^{2d}$, an index orbit defines a \emph{leapfrog orbit} $\{\Phi_h^i(x,v)\}_{i\in I}$, where the total physical time from $\Phi_h^{\min I}(x,v)$ to $\Phi_h^{\max I}(x,v)$ is $h(|I|-1)$.

An index orbit $I$ is said to have the \emph{U-turn property} if
\begin{equation}\label{eq:u-turn}
    v_+\cdot(x_+-x_-)\ <\ 0\quad\text{or}\quad v_-\cdot(x_+-x_-)\ <\ 0\;,
\end{equation}
where $(x_+,v_+)=\Phi_h^{\max I}(x,v)$ and $(x_-,v_-)=\Phi_h^{\min I}(x,v)$. Note that, as shown in Figure~\ref{fig:u-turn}, the U-turn property depends only on the endpoints of the orbit and is not centered at the initial condition $(x,v)$.  This aspect is crucial for ensuring the reversibility of NUTS \cite[Cor 5]{BouRabeeCarpenterMarsden2024}.

Given an index orbit $I$, define $\mathfrak I(I)$ as the collection of index orbits obtained by iteratively halving $I$
\[
\mathfrak{I}(I)\ =\ \bigr\{ I_{j,m}\ :\ j \in \{ 0, 1, \dots, k \}, \, m \in \{ 0, 1, \dots, 2^j - 1\}  \bigr\}
\]
where \( I_{j,m} \) is the \( m \)-th sub-index orbit at the \( j \)-th halving step, defined as
\[
I_{j,m}\ =\ \bigr\{ i \in I\ :\ \min I  + m\,2^{-j}|I|\ \leq\ i\ <\ \min I + (m+1)\,2^{-j}|I|\bigr\}\,.
\]
At each level \( j \), \( I \) is divided into \( 2^j \) sub-index orbits, each of length \( 2^{k-j} \).
 We say that an index orbit $I$ has the \emph{sub-U-turn property} if any index orbit in $\mathfrak I(I)$ has the U-turn property.

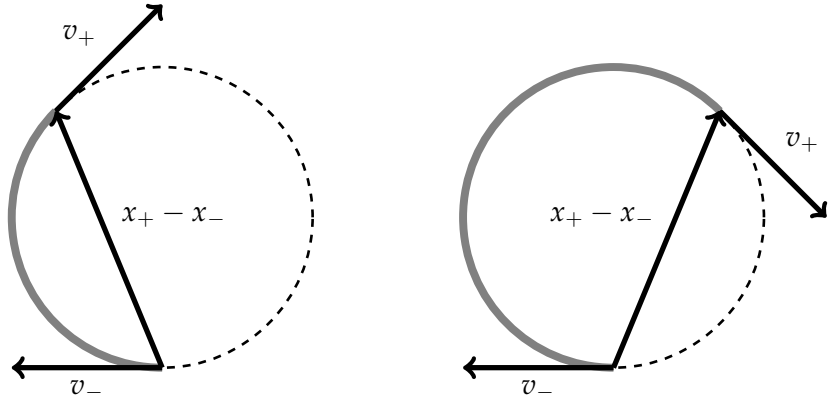
\begin{figure}
\centering
\begin{tikzpicture}[scale=1]
\clip (-0.5,-0.5) rectangle ({7+2*cos(pi/4 r)+sqrt(2)+0.5+1},{2+2*sin(3*pi/4 r)+sqrt(2)+0.5});

\draw[black, line width=1pt, dashed] (2,2) circle (2);
\draw[black, line width=1pt, dashed] (7+1,2) circle (2);

\draw[gray, line width=3pt] (2,0) arc (270:135:2);
\draw[gray, line width=3pt] (7+1,0) arc (270:45:2);

\draw[black, line width=2pt, ->] (2,0) -- (0,0) node[anchor=north, pos=0.5]{$v_-$};
\draw[black, line width=2pt, ->] (7+1,0) -- (5+1,0) node[anchor=north, pos=0.5]{$v_-$};

\draw[black, line width=2pt, ->] ({2+2*cos(3*pi/4 r)},{2+2*sin(3*pi/4 r)}) -- ({2+2*cos(3*pi/4 r)+sqrt(2)},{2+2*sin(3*pi/4 r)+sqrt(2)}) node[anchor=south east, pos=0.5]{$v_+$};
\draw[black, line width=2pt, ->] ({7+2*cos(pi/4 r)+1},{2+2*sin(pi/4 r)}) -- ({7+2*cos(pi/4 r)+sqrt(2)+1},{2+2*sin(pi/4 r)-sqrt(2)}) node[anchor=south west, pos=0.5]{$v_+$};
    
\draw[black, line width=2pt, ->] (2,0) -- ({2+2*cos(3*pi/4 r)},{2+2*sin(3*pi/4 r)}) node[anchor=south west, pos=0.5]{$x_+-x_-$};
\draw[black, line width=2pt, ->] (7+1,0) -- ({7+2*cos(pi/4 r)+1},{2+2*sin(pi/4 r)}) node[anchor=south east, pos=0.5]{$x_+-x_-$};
\end{tikzpicture}
\caption{\textit{Two orbits on $\sqrt d\,\mathcal S^{d-1}$ following the $2\pi$-periodic exact Hamiltonian flow: The left orbit of path length in $[0,\pi]$ does not exhibit the U-turn property \eqref{eq:u-turn}, while the right orbit, with a path length in $(\pi,2\pi)$, does.  In this idealized setting (starting on the $(d-1)$-sphere with tangential velocity of correct magnitude following the exact flow), the U-turn property of an orbit depends only on the orbit's path length and is uniform in the initial position $x$, consistent with the rotational symmetry of the sphere.  In the realistic setting, where $x\in D_\alpha$, $v\sim\gamma$ and  the leapfrog integrator is used, local effects emerge, as shown in Figure \ref{fig:sin}.}}
\label{fig:u-turn}
\end{figure}

\paragraph{Orbit selection.}  Given $(x,v)\in\mathbb R^{2d}$ and a maximal index orbit length of $2^{k_{\mathrm{max}}}$ where $k_{\mathrm{max}} \in \mathbb{N}$, the orbit selection procedure in NUTS proceeds iteratively as follows, starting with $I_0=\{0\}$: 
\begin{itemize}
\item For the current orbit $I_j$, draw an extension $I'$ uniformly from the set $\{I_j-|I_j|,I_j+|I_j|\}$.  
\item If  $I'$ has the sub-U-turn property, stop the process and select $I_j$ as the final orbit.  
\item If $I'$ does not have the sub-U-turn property, extend the orbit by setting $I_{j+1}=I_j\cup I'$.  If $I_{j+1}$ satisfies the U-turn property or  $|I_{j+1}|=2^{k_{\mathrm{max}}}$, stop the process and select $I_{j+1}$ as the final orbit. 
\item Otherwise, repeat the process with $I_{j+1}$ as the current orbit. 
\end{itemize}
This iterative procedure generates a sample $I$ from a probability distribution $\mathcal{O}(x,v)$ defined over the collection of all index orbits containing $0$.

\paragraph{Index selection.}
After sampling an index orbit $I$, an index $L$ is drawn from a Multinoulli distribution, which generalizes the Bernoulli distribution to arbitrary index sets. Specifically, a random variable $L$  follows the Multinoulli distribution  $L \sim\Mult(a_i)_{i\in I}$ with weights $(a_i)_{i\in I}$, if for each $i\in I$, the probability of selecting $i$ is \[
\mathbb P(L=i)\ =\ a_i\,\bigr(\sum\nolimits_{i\in I}a_i\bigr)^{-1} \;. 
\] In NUTS, these sampling probabilities are based on the energies of the corresponding leapfrog iterates. Specifically given a sampled index orbit $I$, the index $L$ is sampled as
\begin{equation} \label{eq:multinoulli} L\ \sim\ \Mult\bigr(e^{-(H\circ\Phi_h^i-H)(x,v)}\bigr)_{i\in I} \;.
\end{equation}

\paragraph{Transition kernel.}  Given the current state $x$, the sampled velocity $v$, and the sampled index $L$, the next state of the NUTS chain is defined as $\Pi(\Phi_h^L(x,v))$, where $\Pi: \mathbb{R}^{2d} \to \mathbb{R}^d$ is the projection onto the position component, i.e., $\Pi(x,v) = x$ for all $(x,v) \in \mathbb{R}^{2d}$.

In summary, given the current state $x \in \mathbb{R}^d$, a step size $h>0$, and a maximum number of doublings $k_{\mathrm{max}} \in \mathbb{N}$,  a complete NUTS transition step is given by:

\begin{algorithm}[H] 
\caption{$X\sim\pi_{\NUTS}(x,\cdot)$}\label{algo:NUTS}
1. Velocity refreshment:  $v\sim\gamma$.\\
2. Orbit selection:  $I\sim\mathcal{O}(x,v)$.\\
3. Index selection:  $L\sim\Mult\bigr(e^{-(H\circ\Phi_h^i-H)(x,v)}\bigr)_{i\in I}$. \\ 4. Output: $X=\Pi( \Phi_h^L(x,v) ) $.
\end{algorithm}

Algorithm~\ref{algo:NUTS} suggests that NUTS can be interpreted as an auxiliary variable method, as defined in \cite[Sec~2.1]{BouRabeeCarpenterMarsden2024}. The auxiliary variables consist of: the velocity $v$, which is sampled independently in the first step; the index orbit $I$, sampled conditionally on the position and velocity in the second step; and the index $L$, sampled conditionally on the preceding variables in the third step of Algorithm~\ref{algo:NUTS}.  After these steps, a Metropolis step is applied in the enlarged space using a deterministic proposal given by the measure-preserving involution: $(x,v,I,L) \mapsto (\mathcal{S} \circ \Phi_h^L(x, v), -(I-L), L)$ where $\mathcal{S}(x,v) = (x,-v)$ \cite[Equation~(25)]{BouRabeeCarpenterMarsden2024}.  Notably,  this proposal is  always accepted in the fourth step, and the auxiliary variables are discarded once used.   By Corollary 5 of \cite{BouRabeeCarpenterMarsden2024},  the following holds.

\medskip

\begin{theorem}\label{thm:rev}
The transition kernel $\pi_{\NUTS}$ is reversible with respect to the target distribution $\mu$. 
\end{theorem}

\paragraph{Locally adaptive HMC.} Given $x \in \mathbb{R}^d$, a Markov kernel $\pi$ is called a \emph{locally adaptive HMC method} if its transitions $X \sim \pi(x,\cdot)$ take the form
\begin{equation}\label{eq:HMCkernel}
    X\ =\ \Pi( \Phi_h^{T/h}(x,v)) \quad\text{with } v\sim\gamma~\text{and}~T\sim\tau_{x,v},
\end{equation}
where $\tau_{x,v}$ is a probability measure on $h\,\mathbb Z$, which may depend on $x$ and $v$.  This  framework covers standard HMC with fixed or randomized path lengths. NUTS also fits within this framework, with $\tau_{x,v}=\law(hL)$, i.e., the distribution of the path length based on the random index $L$.

\subsection{Main Result: Mixing of NUTS}

The spherical geometry of Gaussian concentration is incorporated into the analysis by considering spherical shells of the form
\begin{equation}\label{eq:D_alpha}
	D_{\alpha}\ =\ \bigr\{x \in \mathbb{R}^d \,:\,||x|^2-d|\leq \alpha\bigr\}
\end{equation}
for $\alpha\leq d$.  The Gaussian measure $\gamma$ is highly concentrated in these shells, as quantified by
\begin{equation}\label{eq:Gaussian_HW}
	\gamma(D_\alpha^c)\ \leq\ 2\,\exp(-\alpha^2/8d) \;.
\end{equation}
This result is proven in Lemma \ref{lem:E}, cf. \cite{Vershynin}.

The following theorem provides a mixing time guarantee for NUTS when the target distribution is the high-dimensional canonical Gaussian measure $\gamma$.  Let $\mathcal P(\mathbb R^d)$ denote the set of probability distributions on $\mathbb R^d$ and $2^{\mathbb N}=\{2^n\}_{n\in\mathbb N}$.

\medskip

\begin{theorem}[Main Result] \label{thm:NUTSmixing}
Let $\eps,\alpha_0>0$ and $\nu\in\mathcal P(\mathbb R^d)$ be such that $\max\bigr(\nu(D_{\alpha_0}^c),\,\gamma(D_{\alpha_0}^c)\bigr)\leq\eps/8$.
Let $k_{\mathrm{max}}\in\mathbb N$ and assume $h(2^{k_{\mathrm{max}}}-1)>C_1$ for some absolute constant $C_1>0$.
Treating double-logarithmic factors in $d$ and $\eps^{-1}$ as absolute constants, there exist absolute constants $c_1,C_2,C_3>0$ such that for all
\begin{equation}\label{eq:thm_hbar}
    h\ \leq\ \bar h\ =\ c_1\,\min\bigr(\alpha_0^{-1/2},\,d^{-1/4}\log^{-1/2}d\log^{-3/4}\eps^{-1}\bigr)\log^{-1/2}d
\end{equation}
satisfying the condition
\begin{equation}\label{eq:thm_h_rest}
    h(2^{\mathbb N}-1)\ \cap\ \bigr((0,\delta)\cup(\pi-\delta,\pi+\delta]\bigr)\ =\ \emptyset ~~ \text{with} ~~
    \delta\ =\ C_2\,\max\bigr(d^{-1} \alpha_0,\,d^{-1/2}\log d\log^{3/2}\eps^{-1}\bigr)\;,
\end{equation}
the total variation mixing time of NUTS with respect to the canonical Gaussian measure $\gamma$ starting from $\nu$ to accuracy $\eps$ satisfies
\begin{equation}\label{eq:thm_mix}
    \tmix(\eps,\nu)\ =\ \inf\bigr\{ n \in \mathbb{N}\,:\,\TV\bigr(\nu\pi_{\NUTS}^n,\,\mu\bigr)\leq \eps \bigr\}\ \leq\ C_3\log d\log\eps^{-1}\;.
\end{equation}
\end{theorem}

\begin{remark}[Mixing Time Guarantee]
Theorem~\ref{thm:NUTSmixing} shows that NUTS can achieve an $\eps$-accurate approximation in total variation distance from an initial distribution $\nu$ with
\begin{equation}\label{eq:complexity}
O\Bigr(\max\bigr(\alpha_0^{1/2},\,d^{1/4}\log^{1/2}d\log^{3/4}\eps^{-1}\bigr)\log^{3/2}d\log\eps^{-1}\Bigr) \quad \text{gradient evaluations}
\end{equation}
for $h\propto\bar h$, where $\alpha_0$ defined by $\max\bigr(\nu(D_{\alpha_0}^c),\,\gamma(D_{\alpha_0}^c)\bigr)\leq\eps/8$ quantifies how far $\nu$ and $\gamma$ spread out away from the sphere $\sqrt d\,\mathcal S^{d-1}$. It is at least of order $\sqrt d\log^{1/2}\eps^{-1}$ by \eqref{eq:Gaussian_HW}.  If the initial distribution is a point mass in $ D_{\alpha_0}$ with $\alpha_0=O(\sqrt d\log d\log^{3/2}\eps^{-1})$, \eqref{eq:complexity} simplifies to
\begin{equation}\label{eq:complexity_2}
O\bigr(d^{1/4}\log^2d\log^{7/4}\eps^{-1}\bigr)  \quad \text{gradient evaluations} \;.
\end{equation}
The same holds for any discrete or continuous convex combination of such point masses.  The $d^{1/4}$-scaling is expected to be sharp \cite{BePiRoSaSt2013}.

Comparable mixing time guarantees have been obtained with the conductance approach \cite{LS1993} for well-tuned HMC: in the Gaussian case, $O(d^{1/4}\log^{3/2}(\beta\eps^{-1}))$ gradient evaluations \cite{apers2022}, and in more general settings, assuming an isoperimetric inequality and sufficient regularity, $O((d+\log(\beta\eps^{-1}))^{1/4}\log^{7/4}(\beta\eps^{-1}))$ gradient evaluations \cite{chen2023}.  Here, $\beta=\sup_{B}\nu(B)/\mu(B)$ is the warmness parameter, with the supremum taken over all Borel sets. These conductance-based results are more general than stated and also capture dependencies on the condition number.

A key difference between Theorem \ref{thm:NUTSmixing} and these results, beyond the focus on NUTS instead of HMC, is the dependence on the warmness parameter.  The two mixing time guarantees via conductance only align with the $d^{1/4}$-scaling of \eqref{eq:complexity_2} if $\beta$ scales polynomially in $d$.
In contrast, the mixing time bound from Theorem \ref{thm:NUTSmixing} is independent of $\beta$. In the context of \eqref{eq:complexity_2}, $\beta=\infty$. 

While the conductance method naturally uses isoperimetry to capture the geometric structure of the target distribution, this work demonstrates that a coupling approach, when extended to account for geometric structure, can produce comparable results with less restrictions on the initial distribution. That said, the conductance method has been applied to a much wider range of settings than the Gaussian case considered here.  Nonetheless, this study also paves the way for a conductance-based analysis of NUTS \cite{LS1993, Lovasz99, Chen_Minimax}. 
\end{remark}

\subsection{Discussion of Theorem \ref{thm:NUTSmixing}}

The proof of Theorem \ref{thm:NUTSmixing} is based on the coupling lemma, which states
\begin{equation}
\label{eq:coupling_lemma}
\TV\bigr(\nu\pi_{\NUTS}^n,\,\mu\bigr)\ \leq\ \mathbb P(X_n\ne\widetilde X_n) 
\end{equation}
where $X_n\sim\nu\pi_{\NUTS}^n$ is a NUTS chain initialized in $\nu$ and $\widetilde X_n\sim\mu\pi_{\NUTS}^n=\mu$ is a stationary copy.  These chains are coupled such that the probability of exact meeting increases over the iterations, causing the right hand side to approach zero as $n$ increases.  

In the case of the Gaussian measure $\gamma$, the primary factor driving this meeting is the log-concavity of $\gamma$, which brings the two copies closer when they are synchronously coupled.  Once the copies are sufficiently close, a one-shot coupling can then make them meet exactly \cite{BouRabeeEberle2023}.  However, NUTS’ path length adaptivity complicates this process and its analysis.  

To address this, we establish strong uniformity in NUTS’ transitions within the geometry of  $\gamma$.   This allows us to demonstrate contraction and to implement a one-shot coupling successfully. The uniformity in NUTS’ transitions is achieved by carefully controlling local effects on both orbit and index selection, which is a central theme of this work.  As we discuss further below, Conditions~\eqref{eq:thm_hbar} and \eqref{eq:thm_h_rest} play crucial roles in these arguments.

For the sake of legibility, in this discussion we suppress logarithmic dependencies on $d$ and $\eps^{-1}$ by using the notation: for $\mathsf{x},\mathsf{y}\in\mathbb R$ we write $\mathsf{x}=\widetilde O (\mathsf{y})$ if there exists a constant $C>0$ depending at most logarithmically on $d$ and $\eps^{-1}$ such that $\mathsf{x}\leq C \, \mathsf{y}$.

\paragraph{Localization to a shell.}
We aim to conduct the analysis within a spherical shell $D_\alpha$ with $\alpha=\widetilde O(d^{1/2})$ where the target measure $\gamma$ concentrates as shown by \eqref{eq:Gaussian_HW}.   Setting $\alpha_0=\widetilde O(d^{1/2})$ ensures that both copies of NUTS, appearing in \eqref{eq:coupling_lemma}, start in a comparable shell with sufficient probability.  The stability of NUTS when initialized in this shell then allow us to restrict the analysis to the shell $D_\alpha$.  This approach incorporates the geometric structure of the target measure into the analysis. 

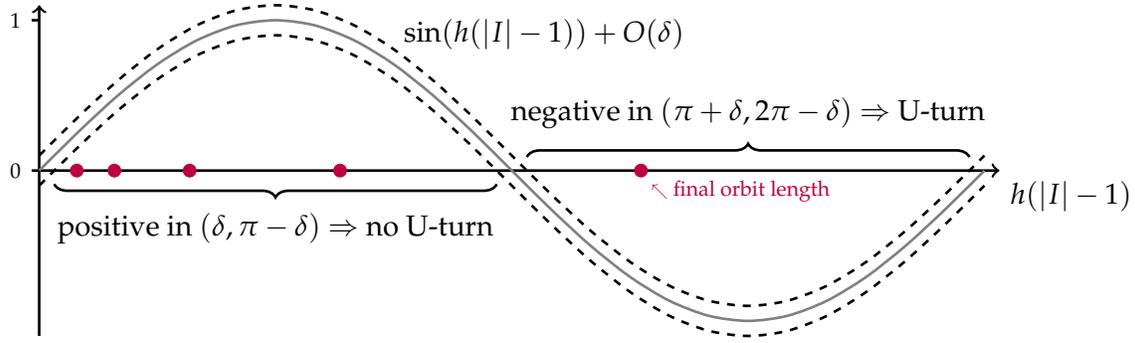
\begin{figure}[t]
\centering
\begin{tikzpicture}[scale=2]
\clip (-0.2,-1.2) rectangle ({2*pi+1},1.2);

\draw[black, line width=1pt, ->] (0,0) -- ({2*pi+0.1},0) node[anchor=north west, pos=1]{$h(|I|-1)$};
\draw[black, line width=1pt, ->] (0,-1.1) -- (0,1.1);
\draw[black, line width=1pt] (-0.05,0) -- (0,0) node[anchor=east, pos=0]{${\scriptstyle 0}$};
\draw[black, line width=1pt] (-0.05,1) -- (0,1) node[anchor=east, pos=0]{${\scriptstyle 1}$};

\draw [gray, line width=1pt] plot[variable=\t,domain=0:2*pi,smooth,thick] ({\t},{sin(\t r)});
\draw [black, line width=1pt, dashed] plot[variable=\t,domain=0:2*pi,smooth,thick] ({\t},{sin(\t r)+0.1});
\draw [black, line width=1pt, dashed] plot[variable=\t,domain=0:2*pi,smooth,thick] ({\t},{sin(\t r)-0.1});

\node (name) at ({3*pi/4},{sin(3*pi/4 r)+0.2}) [anchor=west]{$\sin(h(|I|-1))+O(\delta)$};

\draw [decorate,decoration={brace,amplitude=5pt,mirror,raise=1ex}, line width=1pt]
  ({0.1},0) -- ({pi-0.1},0) node[midway,yshift=-2em]{positive in $(\delta,\pi-\delta)$ $\Rightarrow$ no U-turn};
\draw [decorate,decoration={brace,amplitude=5pt,raise=1ex}, line width=1pt]
  ({pi+0.1},0) -- ({2*pi-0.1},0) node[midway,yshift=2em]{negative in $(\pi+\delta,2\pi-\delta)$ $\Rightarrow$ U-turn};

\filldraw[purple] (0.25,0) circle (1.2pt);
\filldraw[purple] (0.5,0) circle (1.2pt);
\filldraw[purple] (1,0) circle (1.2pt);
\filldraw[purple] (2,0) circle (1.2pt);
\filldraw[purple] (4,0) circle (1.2pt) node[anchor=north west]{\scriptsize ${\scriptstyle \nwarrow}$ final orbit length};

\end{tikzpicture}
\caption{\textit{This figure illustrates \eqref{eq:u-turn_disc}, which shows that the dot products in the U-turn property \eqref{eq:u-turn} are within $O(\delta)$ of a sine function (gray).  The deviations are in particular due to local effects.  For orbit lengths (in physical time) within the highlighted intervals, the deviations do not affect the signs of the dot products, implying the U-turn property is independent of local effects and thus uniform in position $x$.  This implies that the orbit selection consistently finds orbits of uniform length, as described in \eqref{eq:chosenOL_disc}.}}
\label{fig:sin}
\end{figure}

\paragraph{Role of condition \eqref{eq:thm_h_rest}.}
Condition \eqref{eq:thm_h_rest} is crucial for efficient mixing of NUTS.  Without it, the orbit selection procedure may continue iterating to the maximal orbit length.\footnote{In practice, a cap of $2^{10}=1024$ iterations is imposed in NUTS, hence $k_{\mathrm{max}}=10$ in Figures \ref{fig:orbitfig} and \ref{fig:indexfig}.}  This behavior can result in a significant mixing bottleneck, as shown in Figures \ref{fig:orbitfig} (b) and \ref{fig:indexfig} (a).  However, when the step size $h$ satisfies condition~\eqref{eq:thm_h_rest}, something remarkable occurs: the orbit selection step in Algorithm~\ref{algo:NUTS} becomes state-independent within the shell $D_\alpha$ with high probability.  A potential solution to this looping issue is considered in Figure~\ref{fig:fixfig}.

Specifically, in Lemma \ref{lem:u-turn}, we prove that given an initial position $x\in D_\alpha$, an initial velocity $v\sim\gamma$, and any index orbit $I$, the behavior of the U-turn property of $I$ closely resembles the idealized setting shown in Figure~\ref{fig:u-turn}, in the sense that
\begin{equation}\label{eq:u-turn_disc}
    d^{-1}\min\bigr(v_+\cdot(x_+-x_-),\,v_-\cdot(x_+-x_-)\bigr)\ =\ \sin(h(|I|-1))\ +\ O(\delta)\;,
\end{equation}
where $h(|I|-1)$ is the orbit length in physical time and $\delta=\widetilde O(d^{-1/2})$ accounts for small deviations from the sine function, as shown in Figure \ref{fig:sin}.  

As a consequence of \eqref{eq:u-turn_disc}, the U-turn condition in \eqref{eq:u-turn} becomes state-independent within the shell $D_\alpha$ with high probability.  Specifically, with high probability:
\begin{align*}
    h(|I|-1)\ \in\ \bigr[\delta,\,\pi-\delta\bigr]\quad&\Rightarrow\quad\text{The index orbit $I$ does not have the U-turn property.} \\
    h(|I|-1)\ \in\ \bigr(\pi+\delta,\,2\pi-\delta\bigr)\quad&\Rightarrow\quad\text{The index orbit $I$ has the U-turn property.}
\end{align*}
However, outside these intervals, state-dependent local effects emerge that can lead the orbit selection procedure to iterate to the maximum orbit length, as shown in Figures \ref{fig:orbitfig} (b) and \ref{fig:indexfig} (a). 

This looping phenomenon motivates condition \eqref{eq:thm_h_rest}, which ensures that the orbit selection procedure iterates through orbit lengths that fall within the specified intervals, as shown by the red dots in Figure \ref{fig:sin}.  As illustrated in Figures \ref{fig:orbitfig} (a) and (c), when condition \eqref{eq:thm_h_rest} is met, the orbit selection procedure stops at the unique orbit length $2^{k_\ast}$ satisfying:
\begin{equation}\label{eq:chosenOL_disc}
    h(2^{k_\ast}-1)\ \in\ (\pi+\delta,2\pi-\delta)\;.
\end{equation}
Here we realistically assume $h(2^{k_{\mathrm{max}}}-1)>2\pi$.  

For index orbit-length $2^{k_\ast}$, due to the equal probability of doubling either forward or backward in time, the orbit selection procedure generates an index orbit $I$ that is uniformly sampled from the set of index orbits of length $2^{k_\ast}$ containing the starting index $0$.  Specifically, Lemma \ref{prop:u-turn} states that, with high probability, the following holds uniformly in $x\in D_\alpha$:  
\begin{equation}\label{eq:orbitdist_disc}
    \mathcal O(x,v)\ =\ \Unif(\mathfrak I_0(k_\ast)) \;.
\end{equation}

\begin{figure}[t]
\centering
\noindent\begin{minipage}[b]{0.33\textwidth}
	\includegraphics[width=\textwidth]{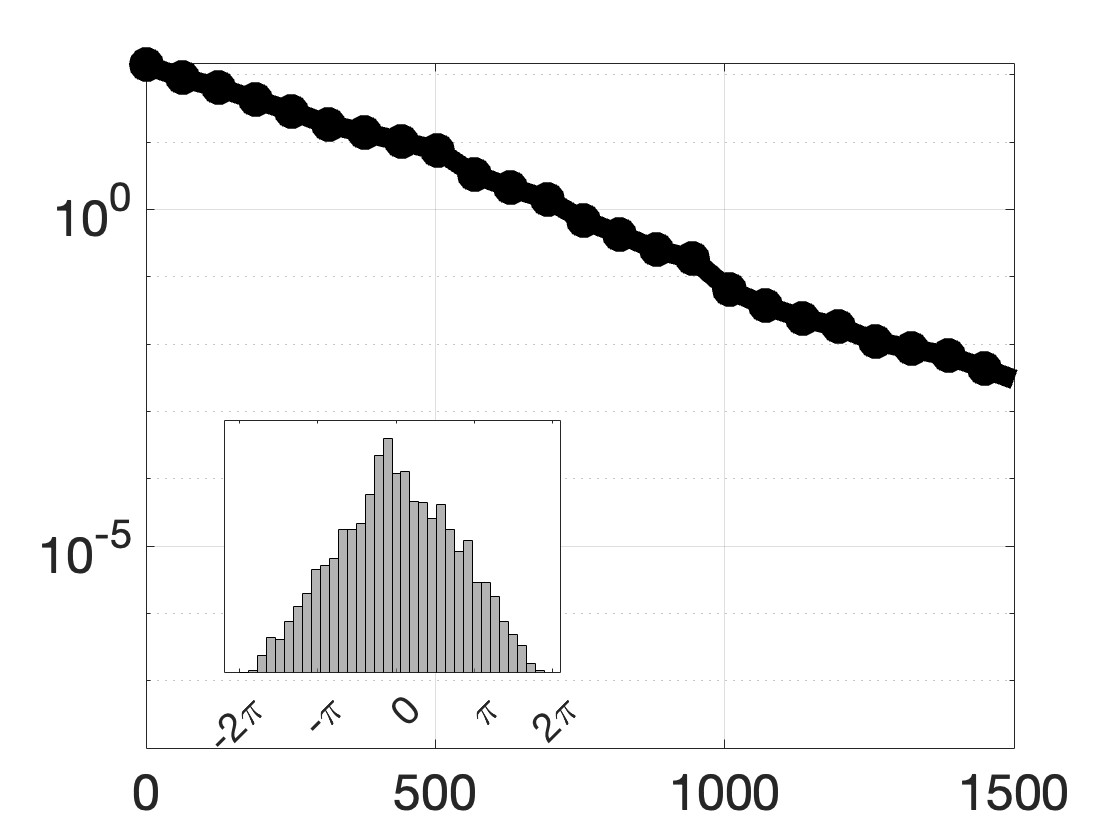}
	\caption*{(a) $h=0.09$}
\end{minipage}%
\noindent\begin{minipage}[b]{0.33\textwidth}
	\includegraphics[width=\textwidth]{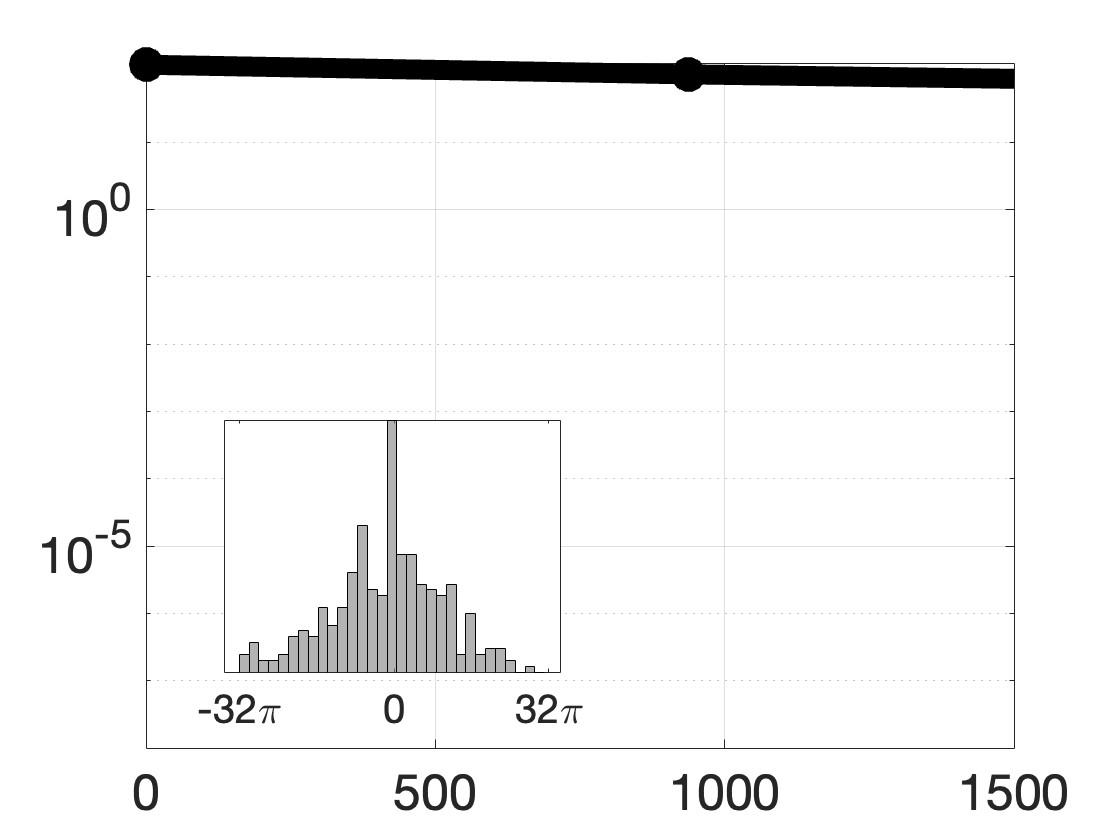}
	\caption*{(b) $h=0.1$}
\end{minipage}%
\noindent\begin{minipage}[b]{0.33\textwidth}
	\includegraphics[width=\textwidth]{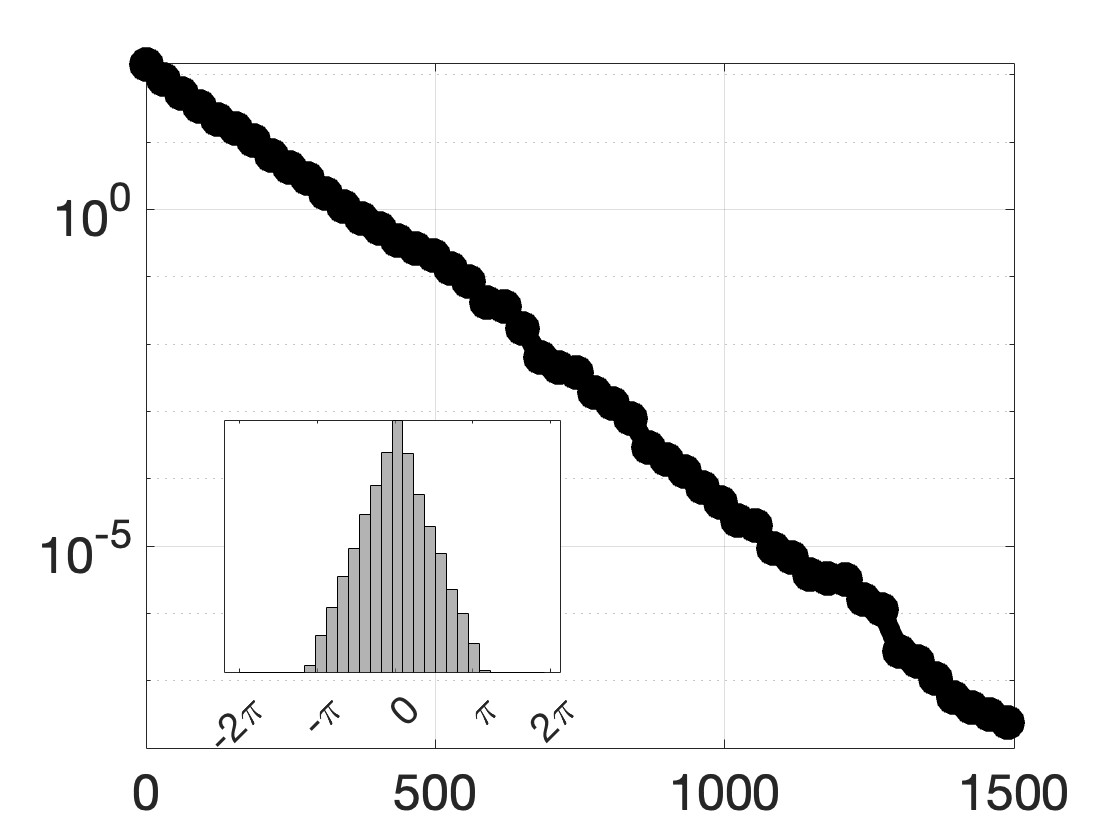}
	\caption*{(c) $h=0.11$}
\end{minipage}%
\caption{\textit{Mean distance between two synchronously coupled copies of NUTS ($k_{\mathrm{max}}=10$) across $100$ realizations in dimension $d=10^4$, plotted against the mean number of leapfrog steps.  All copies are initialized with independent draws from $\gamma$ ($\alpha_0=\widetilde O(d^{1/2})$) and run for $50$ transitions.  All three choices of $h$ satisfy \eqref{eq:thm_hbar} ($\bar h=\widetilde O(d^{-1/4})\approx0.1$). However, in case (b),  condition  \eqref{eq:thm_h_rest} is not met ($0.1(2^5-1)=3.1$ is within a $\delta=\widetilde O(d^{-1/2})\approx0.01$ neighborhood of $\pi$).  Due to local effects on the U-turn property (see Figure \ref{fig:sin}) most transitions in (b) fail to detect any U-turns during orbit selection, leading  to the selection of the maximal orbit length $h(2^{k_{\mathrm{max}}}-1)\approx32\pi$.  In contrast, cases (a) and (c) satisfy condition \eqref{eq:thm_h_rest}, ensuring mixing behavior consistent with Theorem \ref{thm:NUTSmixing}.  In (a) and (b), the final orbit lengths are 5.67 and 3.41, respectively, in physical time consistent with \eqref{eq:chosenOL_disc} (see Figure \ref{fig:sin}).  The insets show histograms of the locally adapted path lengths of NUTS.  We observe approximately triangular distributions in (a) and (c), as expected by \eqref{eq:tau_disc}.  The faster convergence  in (c) compared to (a) is due to the computationally cheaper shorter orbits.}}
\label{fig:orbitfig}
\end{figure}

\paragraph{Role of condition \eqref{eq:thm_hbar}.} Condition~\ref{eq:thm_hbar} is crucial for mixing of NUTS.  Without it, leapfrog energy errors can accumulate, causing the index selection in Algorithm~\ref{algo:NUTS} to break down.  This leads to significant mixing bottlenecks, even if the leapfrog step size meets the stability requirement $h<2$, as shown in Figures~\ref{fig:indexfig} (c) and (d).  

Specifically, for $x\in D_\alpha$ (so that $||x|^2-d|\leq \alpha$ by \eqref{eq:D_alpha}) and $v \sim \gamma$,  the leapfrog energy errors satisfy \begin{equation}\label{eq:deltaH}
    \bigr|H\circ\Phi_h^i-H\bigr|(x,v)
\ =\ \frac{1}{8}h^2\bigr||\Pi(\Phi_h^i(x,v))|^2-|x|^2\bigr|\ =\ \widetilde O(h^2d^{1/2}) \;, \quad i \in \mathbb{Z} \;.
\end{equation}
This results from the leapfrog integrator preserving the modified Hamiltonian $H_h(x,v)=H(x,v)-h^2|x|^2/8$, i.e., $H_h \circ \Phi_h \equiv H_h$, and the entire leapfrog orbit remaining in a comparable shell as $x$ with high probability, as shown in Lemma \ref{lem:stab_leapfrog}.  When condition \eqref{eq:thm_hbar} is met,   $h=\widetilde O(d^{-1/4})$, which controls the energy errors uniformly in $x$.  Given an index orbit $I$, this implies that the index selection procedure samples uniformly from $I$, i.e., $L\sim\Unif(I)$ with fixed probability, see Lemma \ref{lem:MultHMCtoUnifHMC}.

\begin{figure}[t]
\noindent\begin{minipage}[b]{0.25\textwidth}
\includegraphics[width=\textwidth]{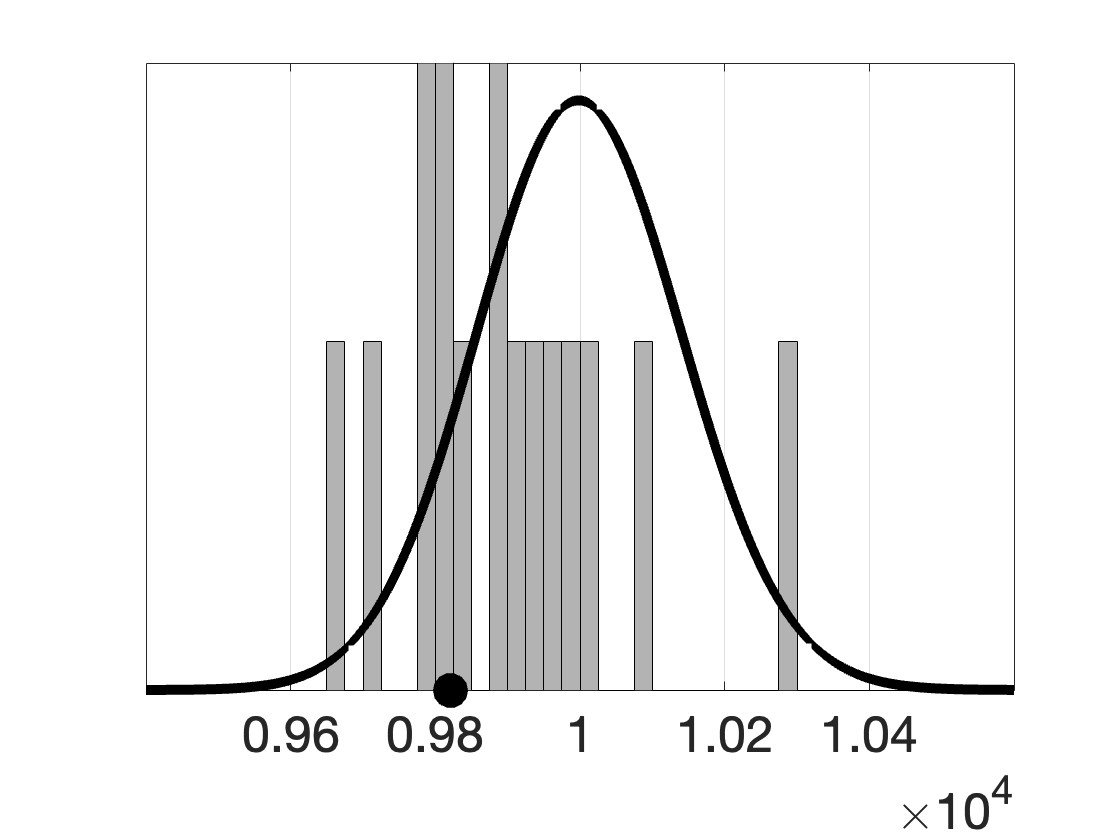}
\caption*{(a) $h=0.1$}
\end{minipage}%
\noindent\begin{minipage}[b]{0.25\textwidth}
\includegraphics[width=\textwidth]{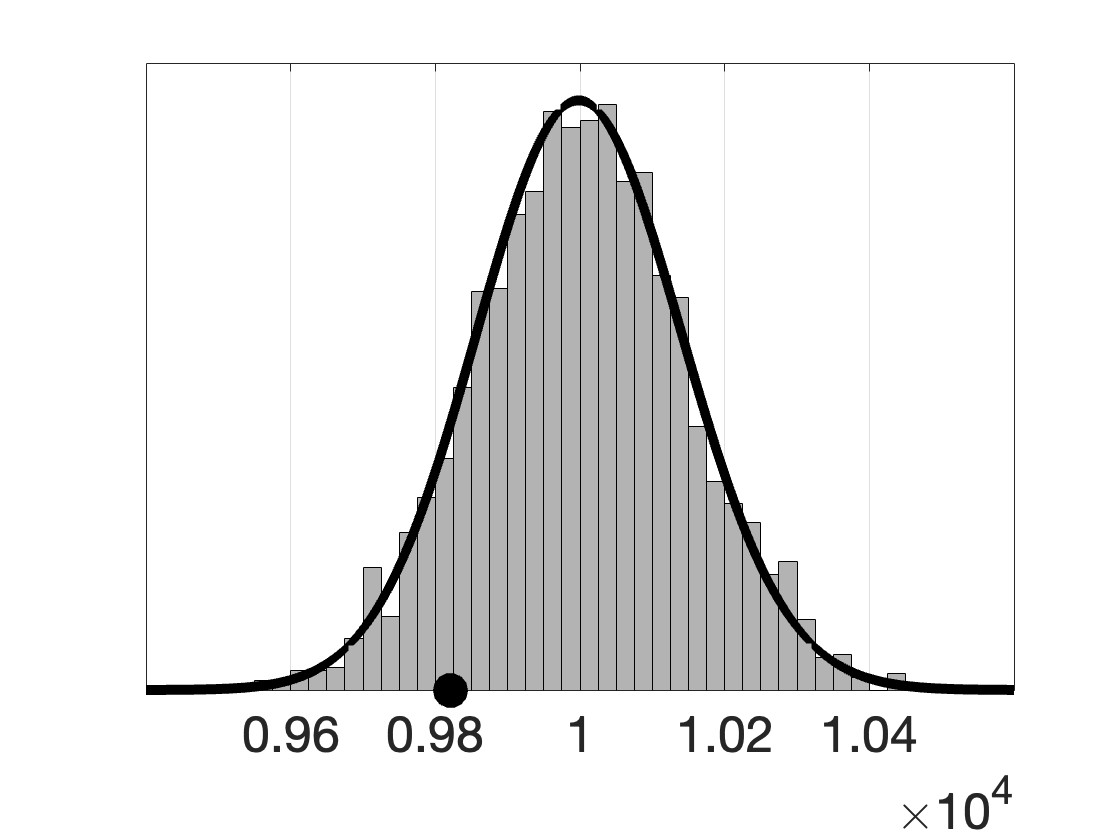}
\caption*{(b) $h=0.11$}
\end{minipage}%
\noindent\begin{minipage}[b]{0.25\textwidth}
\includegraphics[width=\textwidth]{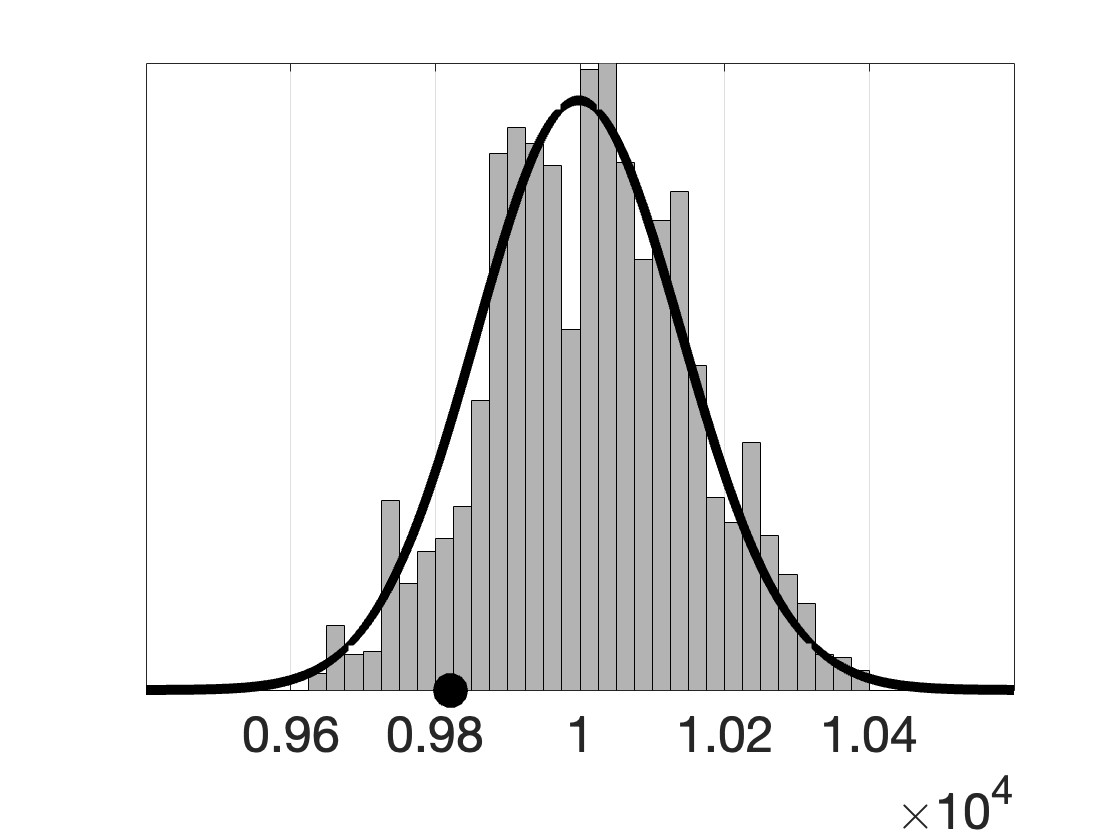}
\caption*{(c) $h=0.5$}
\end{minipage}%
\noindent\begin{minipage}[b]{0.25\textwidth}
\includegraphics[width=\textwidth]{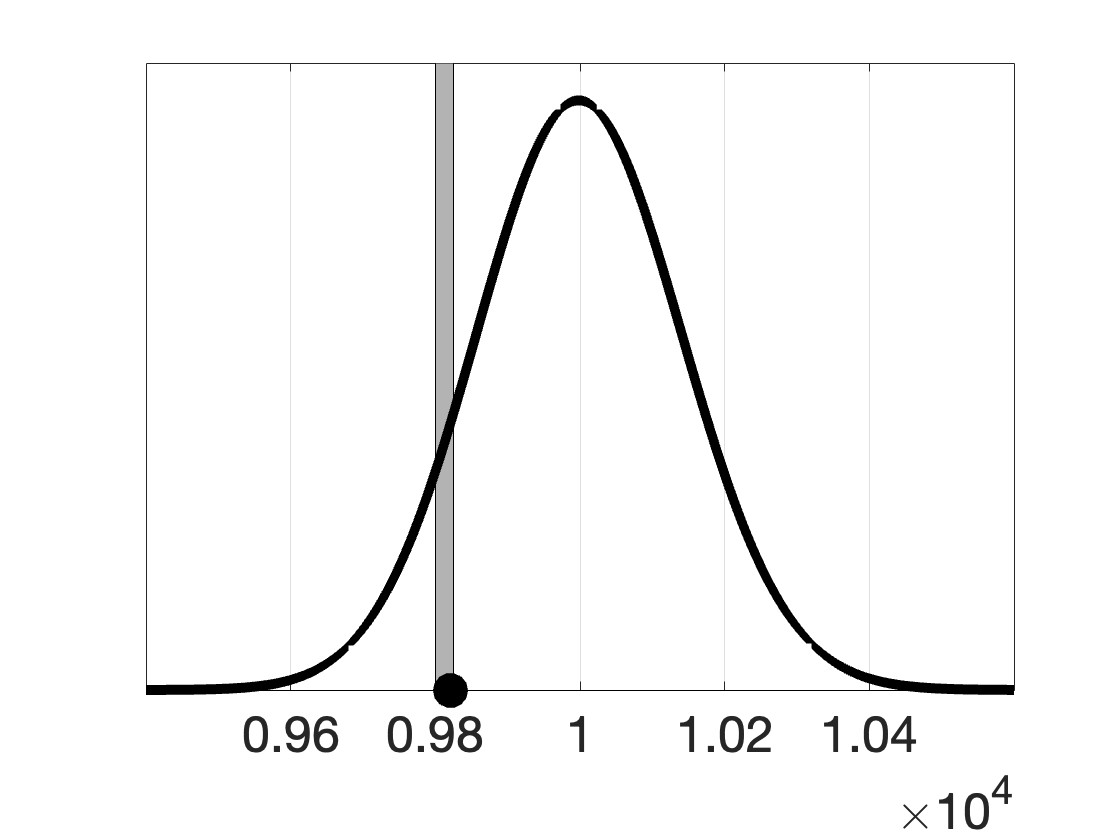}
\caption*{(d) $h=1$}
\end{minipage}%
\caption{\textit{Histograms of $|x|^2$ of NUTS ($k_{\mathrm{max}}=10$) in dimension $d=10^4$ started from a fixed $x_0\sim\gamma$ (marked by a black dot on the x-axis): (a)  over $2$ iterations of $12$ realizations, (b)-(d) over the final $25$ of (b) 50, (c) 660, (d) 2000 iterations of $100$ realizations.  The computational cost of each plot is approximately equal.  The black curves show the exact density of $|x|^2\sim\chi^2(d)$ for $x\sim\gamma$.  (a) shows no mixing within the computational budget due to defective orbit selection, as illustrated in Figure \ref{fig:orbitfig} (b), where condition  \eqref{eq:thm_h_rest} is not met.  (b) shows fast mixing consistent with Theorem \ref{thm:NUTSmixing} and Figure \ref{fig:orbitfig} (c).  (c) and (d) show slow and no mixing, respectively, caused by poor conductance due to large leapfrog energy errors,  where condition \eqref{eq:thm_hbar} is not met.}}
\label{fig:indexfig}
\end{figure}

\paragraph{Comparison to Uniform HMC.}  If the probability of $L\sim\Unif(I)$ is sufficiently large, then combined with \eqref{eq:orbitdist_disc}, the path length distribution of NUTS, as defined in \eqref{eq:HMCkernel}, becomes approximately triangular within the shell $D_{\alpha}$.  Specifically, we have 
\begin{equation}\label{eq:tau_disc}
    \tau_{x,v}(\{t\})\ \approx\ \frac{2^{k_\ast}-|t/h|}{(2^{k_\ast})^2}\quad\text{for $t\in h\,\mathbb Z$ such that $|t|\leq h(2^{k_\ast}-1)$.}
\end{equation}
The insets in Figures \ref{fig:orbitfig} (a) and (c) illustrate this triangular distribution.  Equation \eqref{eq:HMCkernel}, with path length distribution given by the right hand side of \eqref{eq:tau_disc}, defines a transition kernel referred to as \emph{Uniform HMC} named for its uniform orbit and index selection.  This uniformity makes Uniform HMC more tractable for mixing analysis compared to NUTS (see Lemmas \ref{lem:contr} and \ref{lem:OS}).  By comparison, it is then possible to infer mixing of NUTS.  

This comparison motivates the concept of \emph{accept/reject Markov chains} which combine two Markov chains through an accept/reject mechanism.  NUTS fits this framework, following Uniform HMC upon acceptance.  We focus on the \emph{high acceptance regime}, where the accept/reject chain (NUTS) accepts with sufficiently high probability, allowing it to inherit mixing properties of the underlying accept dynamics (Uniform HMC).  This is formalized in Theorem \ref{thm:ARmix} extending a recently introduced coupling framework for the mixing analysis of Metropolized Markov chains \cite{BouRabeeOberdoerster2024}.  

Since both the probability of acceptance and the mixing of the accept chain are generally local properties, the framework further restricts to a subset of the state space by controlling the exit probabilities from this region.  For NUTS, this localization allows us to focus on the shell $D_\alpha$, incorporating the geometry of $\gamma$ into the coupling approach.

\subsection{Open Problems}

Given the limited theoretical work on NUTS, despite its widespread use and importance in modern statistical computing, it seems both timely and necessary to outline directions for future study.

\paragraph{Looping phenomenon.}
NUTS relies on the U-turn property \eqref{eq:u-turn}, a key feature designed to prevent the sampler from inefficiently looping back to its starting point. However, as shown in Figures \ref{fig:orbitfig} (b) and \ref{fig:indexfig} (a), for certain step sizes, NUTS' orbit selection procedure fails to detect U-turns, causing the algorithm to select the maximal orbit length. In practice, this can result in up to $2^{10}$ leapfrog steps per transition, significantly impairing NUTS' efficiency.  A potential solution, as illustrated in Figure \ref{fig:fixfig}, is to use an irregularly spaced time grid instead of an evenly spaced one.

\smallskip

\begin{problem}
Can using a randomly spaced time grid in NUTS' orbit selection procedure relax Condition~\eqref{eq:thm_h_rest}? 
\end{problem}

\paragraph{Mixing time \& asymptotic bias of Uniform HMC.}  Recently, significant progress has been made in quantifying the mixing time of unadjusted samplers \cite{DurmusEberle2021}, such as the Unadjusted Langevin Monte Carlo \cite{durmus2017nonasymptotic,dalalyan2017theoretical,durmus2019high,erdogdu2021convergence}, unadjusted HMC \cite{BouRabeeSchuh2023,BouRabeeEberle2023,monmarche2022hmc}, and unadjusted Kinetic Langevin Monte Carlo \cite{cheng2018underdamped,dalalyan2020sampling,monmarche2021high,monmarche2022hmc}.  These studies provide explicit upper bounds on mixing time and complexity. Although related, Uniform HMC differs from these unadjusted samplers, and there has been little exploration of its complexity or asymptotic bias properties.

\smallskip

\begin{problem}
What are the total variation mixing time and asymptotic bias properties of Uniform HMC for asymptotically strongly log-concave targets and mean-field models? 
\end{problem}

\paragraph{Biased progressive sampling \& the apogee-to-apogee path sampler.}

There are two common methods for index selection in NUTS: (i) Multinoulli sampling, as in \eqref{eq:multinoulli}, and (ii) biased progressive sampling, introduced in \cite{betancourt2017conceptual}. The latter method assigns greater weight to the last doubling in the index selection process, effectively favoring states that are further from the starting point, assuming no looping back. A more systematic approach to this type of weighting is introduced with the apogee-to-apogee path sampler (AAPS) \cite{SherlockUrbasLudkin2023Apogee}, a recently developed locally adaptive HMC method. Like NUTS, AAPS is an auxiliary variable method, and its reversibility is established in \cite[Cor 6]{BouRabeeCarpenterMarsden2024}.

\smallskip

\begin{problem}
What effect does biased progressive sampling have on the mixing time of NUTS?    
\end{problem}

\smallskip

\begin{problem}
How does the mixing time of AAPS depend on its weighting scheme?
\end{problem}

\paragraph{Non-reversible lifts and NUTS.} In a landmark paper,  Diaconis et al~(2000) lifted the reversible symmetric random walk on the discrete circle $\mathbb Z/(n\mathbb Z)$ to $\mathbb Z/(n\mathbb Z)\times\{-1,1\}$ by adding a velocity variable, which flips with probability $\epsilon$ at each step.  On time scales of $\epsilon^{-1}$ steps, the lifted walk moves ballistically, i.e., it explores the circle in a directed manner.  When the time of this ballistic motion becomes comparable to the circumference $n$, the lifted walk mixes in $ O(n)$ steps, while the original symmetric random walk requires at least a diffusive $O(n^2)$ steps \cite{DiHoNe2000}.

Lifting reversible Markov chains and processes to non-reversible counterparts can result in faster convergence to equilibrium leveraging this diffusive-to-ballistic speed-up \cite{chen1999lifting,EberleLoerler24}.
On the one hand, NUTS resembles an unbiased, self-tuning discretization of the Randomized HMC process \cite{BoSa2017,deligiannidis2018randomized,kleppe2022connecting,BoEb2022}, which is a non-reversible lift of the overdamped Langevin diffusion in the sense of \cite{EberleLoerler24} capable of achieving the ballistic speed-up.
On the other hand, as Theorem~\ref{thm:rev} indicates, NUTS is a reversible Markov chain.
This dichotomy between NUTS' resemblance to non-reversible processes and its inherent reversibility naturally leads to the following questions.

\smallskip

\begin{problem}
Can NUTS achieve a diffusive-to-ballistic speed-up, and if so, under what conditions?
\end{problem}

\smallskip

\begin{problem}
If NUTS does not always achieve this speed-up, is there a non-reversible lift of NUTS that can?
\end{problem}

\paragraph{Effect of anisotropy on U-turn property.} As previously discussed, even in the case of an isotropic Gaussian measure, NUTS’ orbit selection procedure is highly sensitive to the choice of step size.  Through our analysis of the U-turn property in \eqref{eq:u-turn} within the shell $D_{\alpha}$, we precisely identified the range of step sizes that can lead to looping behavior (see Figure~\ref{fig:sin}).   However, the behavior of the U-turn property in anisotropic distributions, remains largely unexplored.   Understanding this could provide important insights into how NUTS performs with more complex target distributions.

\medskip

\begin{problem}
How does the U-turn property behave in the context of anisotropic Gaussian distributions, or more generally, for anisotropic log-concave probability measures?      
\end{problem}

\paragraph{Mixing of Riemannian NUTS.} Similar to Riemannian HMC \cite{GiCa2011}, the NUTS algorithm can be extended to Riemannian manifolds by introducing a geometric U-turn property on the tangent bundle \cite{betancourt2013generalizing}.  In this extension, the Riemannian metric adapts to the geometry of the target distribution, acting as a preconditioner for the leapfrog flow on the manifold’s tangent bundle.  By adapting to the local geometry, the preconditioner can improve the conditioning of the problem. This approach is beneficial in  problems with varying curvature: high-curvature regions often require  small leapfrog step sizes for stability, while low-curvature regions benefit from larger path lengths and step sizes. Neal’s funnel problem, with its high curvature in the narrow neck  and low curvature in the wider mouth, is a classic example  \cite{betancourt2013hamiltonianmontecarlohierarchical}.  This leads to the following question.

\smallskip
 
\begin{problem}
What are the mixing properties of Riemannian NUTS in settings with variable curvature, such as Neal’s funnel problem?
\end{problem}

\begin{figure}[t]
\centering
\includegraphics[width=\textwidth]{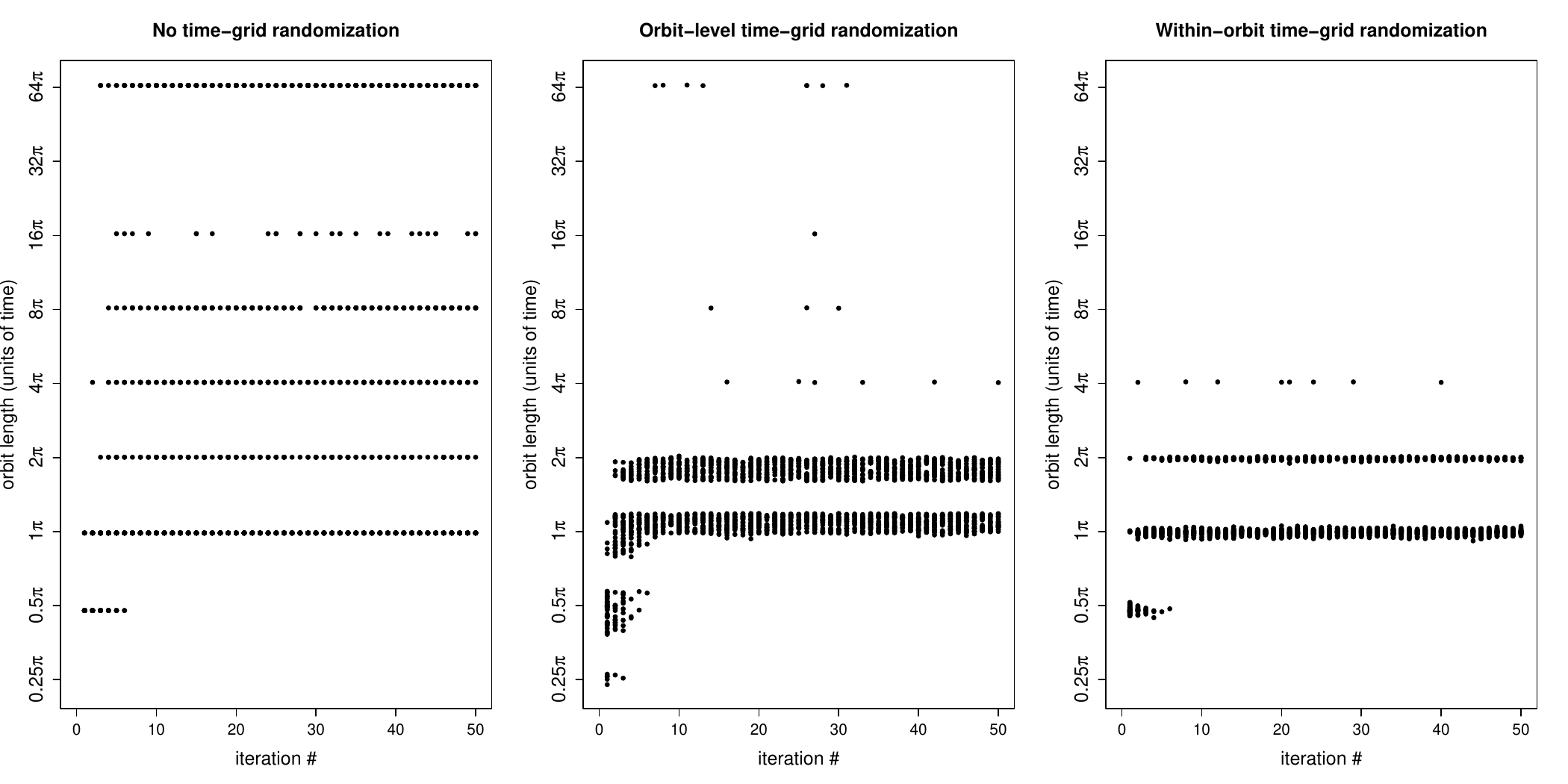}
\caption{\textit{
This figure illustrates a potential solution to NUTS' looping issue in dimension $d=10^4$ using $50$ independent realizations each started from stationarity.  Orbit lengths are plotted against NUTS' iteration number.  (a) uses a fixed leapfrog step size of $h=0.1$ where condition \eqref{eq:thm_h_rest} is not met, as in Figure~\ref{fig:orbitfig} (b). (b) applies step-size randomization at each NUTS transition step.  (c) applies step-size randomization at each leapfrog integration step. (Figure courtesy of and used with permission from Tore Selland Kleppe.)}
}
\label{fig:fixfig}
\end{figure}

\subsection{Paper Outline}

The rest of the paper is organized as follows: In Section \ref{sec:ARchains}, we introduce the notion of accept/reject chains and develop a flexible coupling framework for analyzing their mixing properties.
These tools are later used to compare NUTS with Uniform HMC.  Section \ref{sec:NUTSandGaussianGeometry} establishes the stability required for localization in the shell $D_{\alpha}$, and formalizes the uniformity of NUTS' orbit and index selection steps within this shell, setting the stage for the comparison to Uniform HMC.  In Section \ref{sec:UnifHMC}, we prove mixing time upper bounds for Uniform HMC using coupling techniques.  Finally, in  Section \ref{sec:NUTSproof}, we present the proof of our main result (Theorem \ref{thm:NUTSmixing}).

\subsection*{Open-source code and reproducibility}

Figures~\ref{fig:orbitfig} and \ref{fig:indexfig} can be reproduced using our MATLAB code, which implements NUTS applied to the canonical Gaussian measure on $\mathbb{R}^d$, along with a synchronous coupling of two NUTS chains. This code is available on GitHub under the MIT license.\footnote{Available at \url{https://github.com/oberdoerster/Mixing-of-NUTS}.}

\subsection*{Acknowledgements}

We thank Bob Carpenter, Andreas Eberle, Andrew Gelman, Tore Kleppe, and Francis L\"orler for useful discussions.  N.~Bou-Rabee was partially supported by NSF grant  No. DMS-2111224.  S.~Oberd\"orster was supported by the Deutsche Forschungsgemeinschaft (DFG, German Research Foundation) under Germany’s Excellence Strategy – EXC-2047/1 – 390685813.

\section{Mixing of accept/reject Markov chains}\label{sec:ARchains}

Accept/reject Markov chains combine two Markov kernels through an accept/reject mechanism.
Let $(\Omega,\mathcal A,\mathbb P)$ be a probability space, $S$ a Polish state space with metric $\mathsf{d}$ and Borel $\sigma$-algebra $\mathcal B$, and $\mathcal P(S)$ the set of probability measures on $(S,\mathcal B)$.
For  any $x\in S$, the transition steps $X^{\AR}\sim\pi_{\AR}(x,\cdot)$ of an accept/reject Markov chain take the general form 
\begin{equation} \label{eq:X}
X^{\AR}(\omega)\ =\ \Phi^{\A}(\omega,x)\,\ind_{A(x)}(\omega)\ +\ \Phi^{\R}(\omega,x)\,\ind_{A(x)^c}(\omega)\;, \end{equation}
where $\omega\in\Omega$, $\Phi^{\A},\Phi^{\R}:\Omega\times S\to S$ are product measurable and such that $\Phi^{\A}(\cdot,x)\sim\pi_{\A}(x,\cdot)$ and $\Phi^{\R}(\cdot,x)\sim\pi_{\R}(x,\cdot)$, and the function $A:S\to\mathcal A$ is measurable.  $A(x)$ represents the event of acceptance in the accept/reject mechanism.  In the event of acceptance, the accept/reject chain follows the accept kernel $\pi_{\A}$;  otherwise, it follows the reject kernel $\pi_{\R}$.  For convenience, we write
\[ \pi_{\AR}(x,\cdot)\,\ind_{A(x)}\ =\ \pi_{\A}(x,\cdot)\,\ind_{A(x)}\;. \] 
Between two probability measures $\nu,\eta \in \mathcal{P}(S)$,
the total variation distance and the $L^1$-Wasserstein distance with respect to $\mathsf{d}$ are defined as
\begin{equation} \label{eq:tv_coupling}
\TV\bigr(\nu, \eta \bigr) \ = \ \inf\,\mathbb P[X\neq Y]\quad\text{and}\quad\mathcal W^1_{\mathsf{d}}(\nu,\eta) \ = \ \inf\,\mathbb E\,\mathsf{d}(X,Y)
\end{equation}  
with infima taken over all $\law(X, Y) \in \J(\nu,\eta)$.

\medskip

\begin{theorem}\label{thm:ARmix}
    Let $\eps>0$ be the desired accuracy, $\nu\in\mathcal P(S)$ the initial distribution, $D\subseteq S$, and $\mu\in\mathcal P(S)$ the invariant measure of the accept/reject kernel $\pi_{\AR}$.

    Regarding the accept kernel $\pi_{\A}$, we assume:
    \begin{itemize}
    \item[(i)]
    There exists $\rho>0$ such that for all $x,\tilde{x}\in D$
    \[ \mathcal W^1_{\mathsf{d}}(\pi_{\A}(x,\cdot),\pi_{\A}(\tilde x,\cdot))\ \leq\ (1-\rho)\,\mathsf{d}(x,\tilde x)\;. \]
    
    \item[(ii)]
    There exist $C_{Reg},c>0$ such that for all $x,\tilde{x}\in D$
    \[ \TV\bigr(\pi_{\A}(x,\cdot),\,\pi_{\A}(\tilde x,\cdot)\bigr)\ \leq\ C_{Reg}\,\mathsf{d}(x,\tilde x)\ +\ c\;. \]
    \end{itemize}

    Regarding the probability of rejection, we assume:
    \begin{itemize}
    \item[(iii)]
    There exists an epoch length $\mathfrak E\in\mathbb N$ and $b>0$ such that
    \[ 2\,\mathfrak E\,\sup\nolimits_{x\in D}\mathbb{P}(A(x)^c)\ +\ C_{Reg}\,\diam_{\mathsf{d}}(D)\,\exp\bigr(-\rho(\mathfrak E-1)\bigr)\ +\ b\ \leq\ 1\ -\ c\;. \]
    \end{itemize}

    Regarding the exit probability of the accept/reject chain from $D$, we assume:
    \begin{itemize}
    \item[(iv)]
    Over the total number of transition steps $\mathfrak H=\mathfrak E\,\lceil b^{-1}\log(2/\eps)\rceil$, it holds that
    \[ \mathbb P\bigr(T\:\leq\:\mathfrak H\bigr)\ \leq\ \eps/4 \]
     for both $X^{\AR}_0\sim\nu$ and $X^{\AR}_0\sim\mu$, where $T=\inf\{k\geq0\,:\,X^{\AR}_k\notin D\}$.
    \end{itemize}
    Then, the mixing time of the accept/reject chain satisfies:
    \[ \tmix(\eps,\nu)\ =\ \inf\bigr\{n\geq0\,:\,\TV(\nu\pi_{\AR}^n,\,\mu)\leq\eps\bigr\}\ \leq\ \mathfrak H\;. \]
\end{theorem}

The proof of Theorem~\ref{thm:ARmix} can be straightforwardly adapted from the corresponding result for Metropolized Markov chains introduced in \cite{BouRabeeOberdoerster2024}.  Here, we outline how the assumptions come together in the proof and refer to the reference for a more detailed account.

Assumptions (i) and (ii) can be verified by constructing $\mathcal W^1_{\mathsf{d}}$-contractive  and one-shot couplings of the accept chain $\pi_{\A}$ within the domain $D$.  The contractive coupling reduces the distance between the coupled copies,  while the one-shot coupling allows for exact meeting with a certain probability once the copies are sufficiently close.  This approach, which combines a contractive coupling followed by a one-shot coupling, is a standard technique in the literature \cite{roberts2002one,madras2010quantitative,EbMa2019,monmarche2021high, BouRabeeEberle2023}.

Specifically, over an epoch of $\mathfrak E-1$ contractive transitions followed by one one-shot transition, these assumptions imply the minorization condition
\[ \TV\bigr(\pi_{\A}^{\mathfrak E}(x,\cdot),\,\pi_{\A}^{\mathfrak E}(\tilde x,\cdot)\bigr)\ \leq\ C_{Reg}\,\diam_{\mathsf{d}}(D)\,\exp(-\rho(\mathfrak E-1))\ +\ c \]
within the domain $D$.  If $c<1$, the right-hand side of this inequality is strictly less than $1$ for sufficiently large $\mathfrak E$.  This ensures a nonzero probability of exact meeting between the coupled copies of the accept chain over the epoch.

By isolating the probability of encountering a rejection during the epoch, two copies of the accept/reject chain can be reduced to copies of the accept chain, resulting in
\begin{align*}
    \TV\bigr(\pi_{\AR}^{\mathfrak E}(x,\cdot),\,\pi_{\AR}^{\mathfrak E}(\tilde x,\cdot)\bigr)\ &\leq\ 2\,\mathfrak E\,\sup\nolimits_{x\in D}\mathbb{P}(A(x)^c)\ +\ \TV\bigr(\pi_{\A}^{\mathfrak E}(x,\cdot),\,\pi_{\A}^{\mathfrak E}(\tilde x,\cdot)\bigr) \\
    &\leq\ 2\,\mathfrak E\,\sup\nolimits_{x\in D}\mathbb{P}(A(x)^c)\ +\ C_{Reg}\,\diam_{\mathsf{d}}(D)\,\exp(-\rho(\mathfrak E-1))\ +\ c\;.
\end{align*}
Assumption (iii) ensures that the probability of rejection over sufficiently long epochs is suitably controlled for the right hand side to be bounded above by $1-b\leq e^{-b}$.  By iterating over $\lceil b^{-1}\log(2/\eps)\rceil$ epochs, mixing to accuracy $\eps$ is achieved.  Assumption (iv)  justifies restricting the argument to the domain $D$.

\section{NUTS in the shell}\label{sec:NUTSandGaussianGeometry}

In this section, we analyze the U-turn property, as well as the orbit and index selection of NUTS, within the shells $D_\alpha$ in which the Gaussian measure concentrates.  Additionally, we establish the stability of NUTS in these shells.

\subsection{U-turn property}\label{sec:u-turn}

Let $t_-\leq0\leq t_+$.  In analogy to the U-turn property for leapfrog orbits defined in \eqref{eq:u-turn}, we define the U-turn property for the exact Hamiltonian flow $\phi_t$ over the interval $[t_-, t_+]$.  For any initial condition $(x,v) \in \mathbb{R}^{2d}$, the exact orbit $\{ \phi_t(x,v) \}_{t \in [t_-,t_+]}$ has the U-turn property if
\begin{equation}\label{eq:u-turn_exact}
    v_+\cdot(x_+-x_-)\ <\ 0\quad\text{or}\quad v_-\cdot(x_+-x_-)\ <\ 0\;,
\end{equation}
where $(x_-,v_-)=\phi_{t_-}(x,v)$ and $(x_+,v_+)=\phi_{t_+}(x,v)$.

To built intuition, consider $x\in\sqrt d\,\mathcal S^{d-1}$ with tangential velocity $|v|^2=d$, i.e., $v \cdot x =0$.  Then the position component of the exact Hamiltonian flow 
\begin{equation}\label{eq:HDpos}
	\Pi(\phi_t(x,v))\ =\ \cos t\,x+\sin t\,v 
\end{equation}
is constrained to a $2$-dimensional slice of the $(d-1)$-sphere.
Inserting \eqref{eq:HDpos} into \eqref{eq:u-turn_exact} yields
\begin{equation}\label{eq:u-turn_SPideal}
    v_+\cdot(x_+-x_-)\ =\ \sin(t_+-t_-)\,d\ =\ v_-\cdot(x_+-x_-)\;.
\end{equation}
The dot products in \eqref{eq:u-turn_SPideal} change sign when the orbit length in physical time satisfies $t_+-t_-\in\pi\,\mathbb Z$.  For $t_+-t_-\in[0,\pi]$, the orbit  does not have the U-turn property, whereas for $t_+-t_-\in(\pi,2\pi)$,  the orbit has the U-turn property, as illustrated in Figures \ref{fig:u-turn} and \ref{fig:sin}.

Turning from this idealized setting to the realistic one, we now consider leapfrog orbits starting from $x\in D_\alpha$ with velocity $v\sim\gamma$.  This introduces additional terms in \eqref{eq:u-turn_SPideal}, some of which we control via the following bound based on concentration of $\gamma$.

\medskip

\begin{lemma}\label{lem:E}
	Let $0\leq\alpha,r\leq d$ and $v\sim\gamma$.
    Then the probability of the event
    \[ E_{\alpha,r}\ =\ \Bigr\{v\ :\ \max\bigr(||v|^2-d|,\,\sup\nolimits_{x\in D_\alpha}|x\cdot v|\bigr)\ \leq\ r\Bigr\} \]
    satisfies $\mathbb P(E_{\alpha,r})\geq1-4\,e^{-r^2/8d}$.
\end{lemma}
\begin{proof}[Proof of Lemma \ref{lem:E}]
Let $0\leq\alpha,r\leq d$ and $v\sim\gamma$.  Since $x\in D_\alpha$ satisfies $0 < |x|\leq\sqrt{\alpha+d}\leq\sqrt{2d}$ by \eqref{eq:D_alpha} and \[
|x \cdot v| \ = \ |x| \left| \frac{x}{|x|} \cdot v \right| \ \le \  \sqrt{2d} \left| \frac{x}{|x|} \cdot v \right| \;, \] we have the following bound for $Z\sim\mathcal N(0,1)$,
\[ \mathbb P\bigr(\sup\nolimits_{x\in D_\alpha}|x\cdot v|>r\bigr)\ \leq\ \mathbb P\bigr(|Z|>r/\sqrt{2d}\bigr)\ \leq\ 2\,e^{-r^2/4d}\;. \]

Moreover, the components $v_i$ are i.i.d.~standard normal random variables for $i \in \{1, \dots, d\}$.  Using the exponential Markov inequality, for $0<\lambda\leq1/3$ and $Z\sim\mathcal N(0,1)$, we obtain
\begin{align*}
    \mathbb P(|v|^2-d>r)\ &=\ \mathbb P\Bigr(\sum_{i=1}^d(v_i^2-1)>r\Bigr)\ \leq\ e^{-\lambda r}\,\mathbb Ee^{\lambda\sum_{i=1}^d(v_i^2-1)}\ =\ e^{-\lambda r}\bigr(\mathbb Ee^{\lambda(Z^2-1)}\bigr)^d\\
	&=\ e^{-\lambda r}\bigr(e^{-\lambda}(1-2\lambda)^{-1/2}\bigr)^d\ \leq\ e^{-\lambda r+2\lambda^2d}\;,
\end{align*}
where in the last step we used the inequality $e^{2 \mathsf{x}^2+\mathsf{x}} \sqrt{1-2 \mathsf{x}} \ge 1$ valid for $\mathsf{x} \in [0,1/3]$.
Inserting $\lambda=\min(r/4d,1/3)$, and repeating the argument for $\mathbb P(-(|v|^2-d)>r)$, we conclude that
\[ \mathbb P\bigr(||v|^2-d|>r\bigr)\ \leq\ 2\,e^{-r^2/8d} \]
which completes the proof.
\end{proof}

Using the concentration result from Lemma~\ref{lem:E}, the next lemma extends the uniformity of the U-turn property from the idealized setting in \eqref{eq:u-turn_SPideal} to the realistic scenario in \eqref{eq:u-turn}. This formalization accounts for deviations present in the realistic setting, while preserving the essential uniformity of the U-turn property within the shells with high probability.

\medskip

\begin{lemma}\label{lem:u-turn}
    Let $0\leq\alpha,r\leq d$, $0<h\leq1$ and set
    \begin{equation}\label{eq:u-turn_delta}
    \delta\ =\ \frac\pi2\bigr(5\max(\alpha,r)/d+h^2\bigr)\;.
    \end{equation}
    Let $v\sim\gamma$. In the event $E_{\alpha,r}$ defined in Lemma \ref{lem:E}, for any $x\in D_\alpha$ and any index orbit $I$, 
    \begin{align*}
    h(|I|-1)\ \in\ \bigr[\delta,\,\pi-\delta\bigr]\quad&\Rightarrow\quad\text{The index orbit $I$ does not have the U-turn property.} \\
    h(|I|-1)\ \in\ \bigr(\pi+\delta,\,2\pi-\delta\bigr)\quad&\Rightarrow\quad\text{The index orbit $I$ has the U-turn property.}
    \end{align*}
\end{lemma}

Lemma~\ref{lem:u-turn} demonstrates a striking uniformity in the U-turn property within the shells, but this uniformity holds only if the orbit length in physical time, given by $h (|I|-1)$, lies within the specified intervals. If the step size  is such that $h (2^k-1)$, $k\in\mathbb N$, fails to enter the second interval, possibly due to periodicities in NUTS’ doubling procedure for orbit selection described in Section~\ref{sec:NUTS}, a U-turn might never be detected, as illustrated in Figures~\ref{fig:orbitfig} (b), \ref{fig:indexfig} (a)
and \ref{fig:fixfig}. To prevent this, condition~\eqref{eq:thm_h_rest} on the step size is imposed to rule out the possibility of missed U-turns.  Note that, when $h=\widetilde O(d^{-1/4})$, $\alpha=\widetilde O(d^{1/2})$, and $r=\widetilde O(d^{1/2})$ as in Theorem~\ref{thm:NUTSmixing}, we have $\delta = \widetilde O(d^{-1/2})$. This indicates that the  portion of the above intervals where the $U$-turn property   potentially depends on local effects shrinks at a rate $\widetilde O(d^{-1/2})$.

\begin{proof}[Proof of Lemma \ref{lem:u-turn}]
Let $0<h\leq1$, $0\leq\alpha,r\leq d$, and assume $v\sim\gamma$ to be in $E_{\alpha,r}$ so that \begin{equation}
\label{eq:inEar}
\max\bigr(||v|^2-d|,\,\sup\nolimits_{x\in D_\alpha}|x\cdot v|\bigr)\ \leq\ r \;.
\end{equation}
Let $x\in D_\alpha$ and $I$ be an index orbit such that the corresponding leapfrog orbit $\{\Phi_h^i(x,v)\}_{i\in I}$ has length in physical time $t_+-t_- =h(|I|-1)\leq2\pi$, where \[
t_+=h\max I \quad \text{and} \quad t_-=h\min I \;.
\]
Let $(x_-,v_-)=\Phi_h^{t_-/h}(x,v)$ and $(x_+,v_+)=\Phi_h^{t_+/h}(x,v)$.  

In the standard Gaussian case, we recall that a single step of the underlying leapfrog flow from initial condition $(x,v) \in \mathbb{R}^{2d}$ can be expressed as
\begin{equation}\label{eq:GaussianLF_1}
    \Phi_h(x,v)\ =\ \left(\cos(\beta_hh)x+\frac{\sin(\beta_hh)}{(1-h^2/4)^{1/2}} v,-\sin(\beta_hh)(1-h^2/4)^{1/2}x+\cos(\beta_hh)v\right)
\end{equation}
where we have introduced \[ \beta_h\ =\ \frac{\arccos(1-h^2/2)}{h}\ =\ 1+O(h^2) \quad \text{as  $h  \searrow 0$.}
\] 
In other words, the leapfrog flow corresponds to an elliptical rotation.   Iterating \eqref{eq:GaussianLF_1} for $t/h \in \mathbb Z$ steps results in: \begin{equation}\label{eq:GaussianLF}
    \Phi_h^{t/h}(x,v)\ =\ \left(\cos(\beta_ht)x+\frac{\sin( \beta_h t)}{(1-h^2/4)^{1/2}} v,-\sin(\beta_ht)(1-h^2/4)^{1/2}x+\cos(\beta_ht)v\right) \;. 
\end{equation}

Inserting \eqref{eq:GaussianLF} into the first dot product in the U-turn property \eqref{eq:u-turn}, and applying the triangle inequality, we see that
\begin{align*}
\bigr|v_+\cdot(x_+-x_-)-\sin(t_+-t_-)d\bigr| \ \le \ \rn{1} + \rn{2} + \rn{3} + \rn{4} + \rn{5} 
\end{align*}
where we have introduced: 
\begin{align*}
\rn{1} \ &=\  \bigr|\sin(\beta_h(t_+-t_-))-\sin(t_+-t_-) \bigr| d \;, \\
\rn{2} \ &=\ \left|\left( \frac12\sin(2\beta_ht_+)-\sin(\beta_ht_+)\cos(\beta_ht_-)\right)(1-h^2/4)^{1/2}(|x|^2-d) \right| \;, \\
\rn{3} \ &=\ \left| \left(\frac12\sin(2\beta_ht_+)-\cos(\beta_ht_+)\sin(\beta_ht_-) \right)(1-h^2/4)^{-1/2}(|v|^2-d)\right| \;, \\
\rn{4} \ &= \  \left| \frac12\sin(2\beta_ht_+)\bigr((1-h^2/4)^{-1/2}-(1-h^2/4)^{1/2}\bigr)  +\sin(\beta_ht_+)\cos(\beta_ht_-)\bigr((1-h^2/4)^{1/2}-1\bigr) \right. \\
& \qquad \left. -\cos(\beta_ht_+)\sin(\beta_ht_-)\bigr((1-h^2/4)^{-1/2}-1\bigr) \right| d \;, \\
\rn{5} \ &= \ \bigr| \cos(2\beta_ht_+)-\cos(\beta_h(t_++t_-)) \bigr| \left| x\cdot v \right| \;.
\end{align*}
Since the sine function  is $1$-Lipschitz and $\beta_h>1$, \begin{equation} \label{eq:uturn_rn1}
\rn{1} \ \le \ (\beta_h-1) (t_+-t_-)d \;.
\end{equation}
Using the elementary bound $\left|\sin(\mathsf{a})/2-\sin(\mathsf{b}) \cos(\mathsf{c})\right| \le 3/2$ valid for all $\mathsf{a},\mathsf{b},\mathsf{c} \in \mathbb{R}$, and the bound $\left|\cos(\mathsf{a})-\cos(\mathsf{b})\right| \le 2$ valid for all $\mathsf{a},\mathsf{b} \in \mathbb{R}$, we obtain
\begin{equation} \label{eq:uturn_rn2to5}
\begin{aligned}
\rn{2}+\rn{3} + \rn{4} + \rn{5} \ & \le \ \frac32 \left||x|^2-d \right|+\frac32(1-h^2/4)^{-1/2} \left||v|^2-d \right| \\
& \quad +\frac32\bigr((1-h^2/4)^{-1/2}-(1-h^2/4)^{1/2}\bigr) d + 2|x\cdot v| \;.
\end{aligned}
\end{equation}
Combining \eqref{eq:uturn_rn1} and \eqref{eq:uturn_rn2to5}, and inserting \eqref{eq:inEar}, we obtain
\begin{align*}
\bigr|v_+\cdot(x_+-x_-)-\sin(t_+-t_-)d\bigr| \ \leq\ \frac32(\alpha+r)+h^2d+2r\ \leq\ 5\max(\alpha,r)+h^2d
\end{align*}
which implies that that
\[ v_+\cdot(x_+-x_-)\ =\ \sin(t_+-t_-)d\ +\ R \quad \text{with} \quad |R|\ \leq\ 5\max(\alpha,r)+h^2d\ =\ \frac2\pi\delta d \]
where $R$ depends on $x,v,h,t_+,$ and $t_-$.  An analogous identity holds for the second scalar product in the U-turn property in \eqref{eq:u-turn} and is therefore omitted.

As illustrated in Figure~\ref{fig:sin}, there are two cases that can be controlled:
\begin{description}
\item[Case 1:]
If $t_+-t_-=h(|I|-1)\in[\delta,\pi-\delta]$, then $\sin(t_+-t_-)\geq\frac2\pi\delta$, and both dot products (with different $R$) in the U-turn property are equal to
\[ \sin(t_+-t_-)d+R\ \geq\ \frac2\pi\delta d+R\ \geq\ 0\;, \]
which implies that the index orbit $I$ does not have the U-turn property.
\item[Case 2:]
If $t_+-t_-=h(|I|-1)\in(\pi+\delta,2\pi-\delta)$, then $\sin(t_+-t_-)<-\frac2\pi\delta$ so that both scalar products (with different $R$) in the U-turn property equal
\[ \sin(t_+-t_-)d+R\ <\ -\frac2\pi\delta d+R\ \leq\ 0\;, \]
which implies that the index orbit $I$ has the U-turn property.
\end{description}
This completes the proof.
\end{proof}

\subsection{Orbit selection}

Lemmas \ref{lem:E} and \ref{lem:u-turn} show that, with high probability, there is uniformity in the U-turn property of an index orbit, provided that the step size satisfies condition~\eqref{eq:h_rest} given below. This uniformity implies that with high probability  NUTS selects an index orbit with fixed path length. Despite this, the orbit selection remains stochastic, sampling the uniform distribution over $\mathfrak I_0(k)$, which is the collection of index orbits of the fixed selected length $2^k$ containing $0$.  Consequently, with high probability, NUTS reduces to Multinoulli HMC.  Given $x \in \mathbb{R}^d$ and $k \in \mathbb{N}$, a transition step of Multinoulli HMC is defined as: 

\begin{algorithm}
\caption*{\textbf{Multinoulli HMC:} $X\sim\pi_{\mathrm{MultHMC}}(x,\cdot)$}
1. Velocity refreshment:  $v\sim\gamma$.\\
2. Orbit selection:  $I\sim\Unif(\mathfrak I_0(k))$.\\
3. Index selection:  $L\sim\Mult\bigr(e^{-(H\circ\Phi_h^i-H)(x,v)}\bigr)_{i\in I}$. \\
4. Output: $X=\Pi(\Phi_h^L(x,v))$.
\end{algorithm}

\smallskip

In the next lemma, we formally establish that, with high probability, NUTS in the shells reduces to a specific instance of Multinoulli HMC.  It is important to note that this lemma depends on the step-size restriction in \eqref{eq:h_rest}; without this restriction, the orbit selection procedure may fail to detect a U-turn, as illustrated in Figure~\ref{fig:orbitfig} (b).
\medskip

\begin{lemma}\label{prop:u-turn}
Let $0\leq\alpha,r\leq d$, $0<h\leq1$, and let $\delta$ be as defined in \eqref{eq:u-turn_delta}. Assume the following step-size condition holds:
\begin{equation}\label{eq:h_rest}
    h(2^{\mathbb N}-1)\ \cap\ \bigr((0,\delta)\cup(\pi-\delta,\pi+\delta]\bigr)\ =\ \emptyset\;.
\end{equation}
For $x\in D_\alpha$ and $v\sim\gamma$, within the event $E_{\alpha,r}$ defined in Lemma \ref{lem:E}, we have
\begin{equation}\label{eq:OtuUnif}
    \mathcal{O}(x,v)\ =\ \Unif\bigr(\mathfrak I_0(\min(k_\ast,k_{\mathrm{max}})\bigr)\;,
\end{equation}
where $k_\ast$ is the unique integer such that $h(2^{k_\ast}-1)\in(\pi+\delta,2\pi-\delta)$.

Thus, in the event $E_{\alpha, r}$, NUTS simplifies to Multinoulli HMC with index orbit length $2^{\min(k_\ast,k_{\mathrm{max}})}$, leading to the identity
\begin{equation}\label{eq:NUTStoMultHMC}
    \pi_{\NUTS}(x,\cdot)\,\ind_{E_{\alpha,r}}\ =\ \pi_{\mathrm{MultHMC}}(x,\cdot)\,\ind_{E_{\alpha,r}}\;.
\end{equation}
\end{lemma}

The assertion \eqref{eq:NUTStoMultHMC} follows immediately from \eqref{eq:OtuUnif} by definition of Multinoulli HMC.

\begin{proof}
Let $(x,v)\in\mathbb R^{2d}$.
The orbit distribution $\mathcal{O}(x,v)$ follows from the orbit selection procedure described in Section \ref{sec:NUTS}:
Starting from $I_0=\{0\}$, the procedure iteratively draws an extension $I'$ uniformly from $\{I_k-|I_k|,I_k+|I_k|\}$ and continues until one of the following conditions is met: 
\begin{itemize}
\item The extension $I'$ has a sub-U-turn property. \item The extended orbit $I_{k+1}=I_k\cup I'$ has the U-turn property, or $|I_{k+1}|=2^{k_{\mathrm{max}}}$.  
\end{itemize} If the extension $I'$ has the sub-U-turn property, then $I_k$ is selected as the final orbit.  Alternatively, if the extended index orbit $I_{k+1}$  has the U-turn property or reaches the maximal length $2^{k_{\mathrm{max}}}$, then $I_{k+1}$ is selected as the final orbit.

In particular, if the U-turn property depends only on the orbit length, then the extension $I'$ will not have the sub-U-turn property since none of the index orbits derived from the extensions $I'$ by repeated halving will have the U-turn property, as these orbit have already been negatively tested for the U-turn property in previous iterations.  Therefore, the procedure selects the first extension $I_{k+1}$ of length $2^{k_\ast}$ that satisfies the U-turn property or, alternatively, selects the orbit where $|I_{k+1}|=2^{k_{\mathrm{max}}}$.  Since the extensions are performed in both directions with equal probability, this results in a uniform distribution over the set $\mathfrak I_0(\min(k_\ast,k_{\mathrm{max}}))$, which consists of index orbits of length $2^{\min(k_\ast,k_{\mathrm{max}})}$ that contain $0$, i.e.,
\[ \mathcal{O}(x,v)\ =\ \Unif\bigr(\mathfrak I_0(\min(k_\ast,k_{\mathrm{max}}))\bigr)\;. \]
It therefore suffices to show that the U-turn property depends only on orbit length for step sizes that satisfy condition \eqref{eq:h_rest}.

Let $0\leq\alpha,r\leq d$, $x\in D_\alpha$, and assume $v\sim\gamma$ to be in $E_{\alpha,r}$.  According to Lemma \ref{lem:u-turn}, the U-turn property depends only on orbit length for orbits whose length falls within the specified intervals. Condition \eqref{eq:h_rest} on the step size ensures that the procedure iterates through orbits  until the first orbit of length $2^{k_\ast}$ is found such that
\[ h(2^{k_\ast}-1)\ \in\ (\pi+\delta,2\pi-\delta) \]
at which point the orbit has the U-turn property, as established by the second implication of Lemma \ref{lem:u-turn}. This process is illustrated by the red dots in Figure~\ref{fig:sin}.  If no such orbit is encountered after $k_{\mathrm{max}}$ doublings, the procedure selects an orbit of length $2^{k_{\mathrm{max}}}$.
\end{proof}

\subsection{Index selection}

In the previous section, we showed that NUTS simplifies to Multinoulli HMC in the shell with high probability.  While the orbit selection in Multinoulli HMC is  independent of the state, the Boltzmann-weighted index selection depends on the leapfrog energy errors along the selected leapfrog path.  However, these leapfrog energy errors can be upper bounded uniformly within this shell.  This uniformity in the energy error motivates writing Multinoulli HMC as an accept/reject chain in the sense of Section \ref{sec:ARchains} where the acceptance chain corresponds to Uniform HMC, which is characterized by uniform orbit and index selection.  Given $x \in \mathbb{R}^d$ and $k \in \mathbb{N}$, a transition step of Uniform HMC can be described as:

\begin{algorithm}
\caption*{\textbf{Uniform HMC:} $X\sim\pi_{\mathrm{UnifHMC}}(x,\cdot)$}
1. Velocity refreshment:  $v\sim\gamma$.\\
2. Orbit selection:  $I\sim\Unif(\mathfrak I_0(k))$.\\
3. Index selection:  $L\sim\Unif(I)$. \\
4. Output: $X=\Pi(\Phi_h^L(x,v))$.
\label{alg:UnifHMC}
\end{algorithm}

\medskip

\begin{lemma}\label{lem:MultHMCtoUnifHMC}
Let $U\sim\Unif([0,1])$.  Given $(x,v) \in \mathbb{R}^{2d}$ and an index orbit $I \subset \mathbb{Z}$,  define the event
\begin{equation*}
    A_{\mathrm{index}}(x)\ =\ \Bigr\{U\ \leq\ |I|\min_{i\in I}e^{-(H\circ\Phi_h^i-H)(x,v)}\Bigr(\sum_{i\in I}e^{-(H\circ\Phi_h^i-H)(x,v)}\Bigr)^{-1}\Bigr\}\;.
\end{equation*}
Then, Multinoulli HMC can be written as
\[ \pi_{\mathrm{MultHMC}}(x,\cdot)\,\ind_{A_{\mathrm{index}}(x)}\ =\ \pi_{\mathrm{UnifHMC}}(x,\cdot)\,\ind_{A_{\mathrm{index}}(x)}\;. \]
\end{lemma}

\begin{proof}
Given $(x,v) \in \mathbb{R}^{2d}$ and an index orbit $I \subset \mathbb{Z}$,  the index $L$ in Multinoulli HMC can be sampled as follows (see Figure \ref{fig:Multinoulli_splitting}) : First, independently draw
\[ L_{\A}\sim\Unif(I)\quad\text{and}\quad L_{\R}\sim\Mult\bigr(e^{-(H\circ\Phi_h^i-H)(x,v)}-\min\nolimits_{i\in I}e^{-(H\circ\Phi_h^i-H)(x,v)}\bigr)_{i\in I} \]
then select
\[ L\ =\ L_{\A}\,\ind_{A_{\mathrm{index}}(x)}+L_{\R}\,\ind_{A_{\mathrm{index}}(x)^c}\;. \]

\begin{figure}[t]
\centering
\begin{tikzpicture}[scale=1]
\clip (-2,-1) rectangle (8,5);

\fill[gray!50] (1,0) rectangle (1.5,3);
\fill[purple!50] (1,3) rectangle (1.5,4);
\fill[gray!50] (2,0) rectangle (2.5,3);
\fill[purple!50] (2,3) rectangle (2.5,3.7);
\fill[gray!50] (3,0) rectangle (3.5,3);
\fill[purple!50] (3,3) rectangle (3.5,3.5);
\fill[gray!50] (4,0) rectangle (4.5,3);
\fill[purple] (4,3) rectangle (4.5,3);
\fill[gray!50] (5,0) rectangle (5.5,3);
\fill[purple!50] (5,3) rectangle (5.5,3.5);
\fill[gray!50] (6,0) rectangle (6.5,3);
\fill[purple!50] (6,3) rectangle (6.5,3.8);

\draw[black, line width=1pt, ->] (0,0) -- (0,4.5);
\draw[black, line width=1pt] (-0.05,0) -- (0,0) node[anchor=east, pos=0]{$0$};
\draw[black, line width=1pt] (-0.05,4) -- (0,4) node[anchor=east, pos=0]{$a_{i_1}$};
\draw[black, line width=1pt] (-0.05,3.7) -- (0,3.7) node[anchor=east, pos=0]{$a_{i_2}$};
\draw[black, line width=1pt] (-0.05,3) -- (0,3) node[anchor=east, pos=0]{$\min_{i\in I}a_i$};
\draw[black, line width=1pt, dashed] (0,4) -- (1,4);
\draw[black, line width=1pt, dashed] (0,3.7) -- (2,3.7);
\draw[black, line width=2pt] (0,3) -- (6.5,3);

\draw[black, line width=1pt] (1,0) rectangle (1.5,4) node[anchor=north, pos=0]{$i_1$};
\draw[black, line width=1pt] (2,0) rectangle (2.5,3.7) node[anchor=north, pos=0]{$i_2$};
\draw[black, line width=1pt] (3,0) rectangle (3.5,3.5) node[anchor=north, pos=0]{$i_3$};
\draw[black, line width=1pt] (4,0) rectangle (4.5,3) node[anchor=north, pos=0]{$i_4$};
\draw[black, line width=1pt] (5,0) rectangle (5.5,3.5) node[anchor=north, pos=0]{$i_5$};
\draw[black, line width=1pt] (6,0) rectangle (6.5,3.8) node[anchor=north, pos=0]{$i_6$};

\end{tikzpicture}
\caption{\textit{For an index set $I$ with $|I|=6$, the Multinoulli distribution $\Mult(a_i)_{i\in I}$ with normalized weights $\sum_{i\in I}a_i=1$ is split into its maximal uniform part $|I|\min_{i\in I}a_i\,\Unif(I)$ (shown in gray) and the remaining part $(1-|I|\min_{i\in I}a_i)\,\Mult(a_i-\min_{i\in I}a_i)_{i\in I}$ (shown in red).}}
\label{fig:Multinoulli_splitting}
\end{figure}
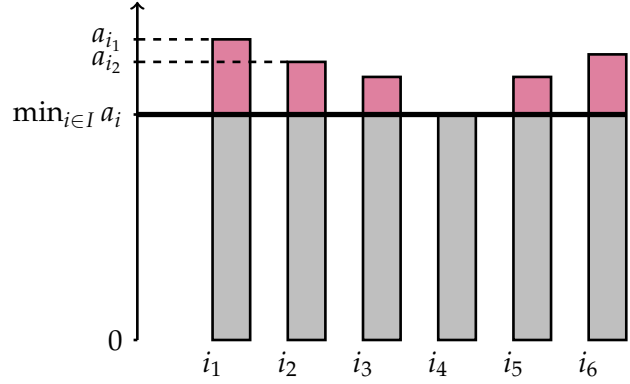

In the event $A_{\mathrm{index}}(x)$, the index selection becomes uniform by construction, and Multinoulli HMC corresponds to Uniform HMC.
\end{proof}

Note that Uniform HMC features fully state independent orbit and index selection, which makes it more straightforward to analyze.  By combining the reduction of NUTS to Multinoulli HMC from Lemma \ref{prop:u-turn} with the previous lemma, we can write NUTS as an accept/reject chain where the acceptance chain corresponds to Uniform HMC.  This shows that NUTS fits the framework of Theorem \ref{thm:ARmix}, and therefore,  the mixing of NUTS in the shell reduces to the mixing of Uniform HMC.  

\medskip

\begin{lemma}\label{prop:NUTStoUnifHMC}
Let $0\leq\alpha,r\leq d$, $\delta$ as defined in \eqref{eq:u-turn_delta} and $0<h\leq1$ such that \eqref{eq:h_rest} holds.
Let $x\in D_\alpha$ and $v\sim\gamma$.
Then, in the event $E_{\alpha,r}\cap A_{\mathrm{index}}(x)$, NUTS simplifies to Uniform HMC with fixed index orbit length of $2^{\min(k_\ast,k_{\mathrm{max}})}$, i.e., 
\begin{equation}\label{eq:NUTStoUnifHMC}
    \pi_{\NUTS}(x,\cdot)\,\ind_{E_{\alpha,r}\cap A_{\mathrm{index}}(x)}\ =\ \pi_{\mathrm{UnifHMC}}(x,\cdot)\,\ind_{E_{\alpha,r}\cap A_{\mathrm{index}}(x)}\;,
\end{equation}
where $E_{\alpha,r}$ is defined in Lemma \ref{lem:E}, $A_{\mathrm{index}}(x)$ is defined in Lemma \ref{lem:MultHMCtoUnifHMC}, and $k_\ast$ is given in Lemma \ref{prop:u-turn}.  Furthermore, the rejection probability is bounded by:
\begin{equation}\label{eq:rej_prob}
    \mathbb P\bigr((E_{\alpha,r}\cap A_{\mathrm{index}}(x))^c\bigr)\ \leq\ 4\,e^{-r^2/8d}+h^2\max(\alpha,r)+\frac14h^4d\;.
\end{equation}
\end{lemma}

\begin{proof}
The assertion in \eqref{eq:NUTStoUnifHMC}  follows directly from Lemma \ref{prop:u-turn} and Lemma  \ref{lem:MultHMCtoUnifHMC}, using transitivity.  For \eqref{eq:rej_prob}, observe that
\begin{equation}\label{eq:rej_prob_sum}
    \mathbb P\bigr((E_{\alpha,r}\cap A_{\mathrm{index}}(x))^c\bigr)\ =\ \mathbb P(E_{\alpha,r}^c)+\mathbb P(E_{\alpha,r}\cap A_{\mathrm{index}}(x)^c)\;.
\end{equation}
Let $0\leq\alpha,r\leq d$, $x\in D_\alpha$ and $v\sim\gamma$.
By Lemma \ref{lem:E}, the first term in \eqref{eq:rej_prob_sum} is bounded by $4\,e^{-r^2/8d}$.

For the second term, we use the definition of $A_{\mathrm{index}}(x)$, which gives
\begin{align*}
    \mathbb P\bigr(A_{\mathrm{index}}(x)^c|v,I\bigr)
\ &=\ 1-|I|\min\nolimits_{i\in I}e^{-(H\circ\Phi_h^i-H)(x,v)}\Bigr(\sum_{i\in I}e^{-(H\circ\Phi_h^i-H)(x,v)}\Bigr)^{-1} \\
&\leq\ 1-e^{-2\sup\nolimits_{i\in I}|H\circ\Phi_h^i-H|(x,v)}
\ \leq\ 2\sup\nolimits_{i\in\mathbb Z}\bigr|H\circ\Phi_h^i-H\bigr|(x,v) \;.
\end{align*}
Inserting this bound into the second term of \eqref{eq:rej_prob_sum} gives
\begin{equation}  \label{eq:rej_prob_summand_2}
\mathbb P\bigr(E_{\alpha,r}\cap A_{\mathrm{index}}(x)^c\bigr)\ =\ \mathbb E\bigr(\mathbb P\bigr(A_{\mathrm{index}}(x)^c|v,I\bigr)\,\ind_{E_{\alpha,r}}\bigr)\ \leq\ 2\,\mathbb E\bigr(\sup\nolimits_{i\in\mathbb Z}\bigr|H\circ\Phi_h^i-H\bigr|(x,v)\,\ind_{E_{\alpha,r}}\bigr)\;. \end{equation}
From \eqref{eq:deltaH}, the definition of $D_{\alpha}$ in \eqref{eq:D_alpha}, and the triangle inequality, we know the energy error satisfes:
\[ \bigr|H\circ\Phi_h^i-H\bigr|(x,v)
\ =\ \frac{h^2}{8}\bigr||\Pi(\Phi_h^i(x,v))|^2-|x|^2\bigr|
\ \leq\ \frac{h^2}{8}\Bigr(\bigr||\Pi(\Phi_h^i(x,v))|^2-d\bigr|+\alpha\Bigr)\;. \]
Now, inserting the explicit leapfrog flow from \eqref{eq:GaussianLF} and using the estimates \[ \max\bigr(||v|^2-d|,\,\sup\nolimits_{x\in D_\alpha}|x\cdot v|\bigr)\ \leq\ r
\] which are valid in $E_{\alpha,r}$, and applying the elementary bound \[
1+\mathsf{a}^2/3 \ge (1-\mathsf{a}^2/4)^{-1} \quad \text{for $\mathsf{a} \in (0,1]$} \;,
\] we obtain
\begin{align}\label{eq:Phi2-d}
	\bigr||\Pi(\Phi_h^{t/h}(x,v))|^2-d\bigr|\ &\leq\
	\begin{aligned}[t]
		&\cos^2(\beta_ht)\,\bigr||x|^2-d\bigr|+\sin^2(\beta_ht)(1-h^2/4)^{-1}\bigr||v|^2-d\bigr|\\
		&+\sin^2(\beta_ht)\bigr((1-h^2/4)^{-1}-1\bigr)d+|\sin(2\beta_ht)|(1-h^2/4)^{-1/2}|x\cdot v|
    \end{aligned}\nonumber\\
	&\leq\ \cos^2(\beta_ht)\alpha+\sin^2(\beta_ht)(1+h^2/3)r+h^2d/3+(1+h^2/3)r \nonumber\\
	&\leq\ \max(\alpha,r)+r+h^2d\;.
\end{align}
Therefore, uniformly in $i \in \mathbb{Z}$,
\[ \bigr|H\circ\Phi_h^i-H\bigr|(x,v)\ \leq\ \frac12h^2\max(\alpha,r)+\frac18h^4d \;. \]
Inserting this bound into \eqref{eq:rej_prob_summand_2}  completes the proof.
\end{proof}

\subsection{Stability of locally adaptive HMC}

In order to localize to the spherical geometry of $\gamma$ in the framework of Theorem \ref{thm:ARmix}, we now explore how NUTS and more generally locally adaptive HMC kernels of the form \eqref{eq:HMCkernel} are stable in spherical shells $D_\alpha$.  Let $T_\alpha$ be the first exit time for the chain starting in $D_\alpha$, i.e.,
\[ T_\alpha\ =\ \min\{k\geq0\,:\,X_k\notin D_\alpha\}\;. \]

For a transition from $x\in D_\alpha$, it holds in $E_{\alpha,r}$ defined in Lemma \ref{lem:E} by \eqref{eq:Phi2-d} that
\begin{equation*}\label{eq:Phi2-d_2}
    \bigr||\Pi(\Phi_h^{t/h}(x,v))|^2-d\bigr|\ \leq\ \max(\alpha,r)+r+h^2d\;.
\end{equation*}
Therefore,
\begin{equation}\label{eq:LFsubset}
    \Pi(\Phi_h^{T/h}(D_\alpha,v)) \ \subseteq \ D_{\max(\alpha,r)+r+h^2d}\;,
\end{equation}
which immediately implies stability in growing spherical shells.

\medskip

\begin{lemma}\label{lem:stab_leapfrog}
	Let $\pi$ be of the form \eqref{eq:HMCkernel} with invariant distribution $\gamma$.
	Let $\alpha_0,r\geq0$.
	For any starting point $x_0\in D_{\alpha_0}$, the exit probability of $X_k\sim\pi^k(x_0,\cdot)$ from $D_{\max(\alpha_0,r)+n(r+h^2d)}$ over $n$ transitions such that $\max(\alpha_0,r)+(n-1)(r+h^2d)\leq d$ satisfies
		\[ \mathbb P_{x_0}(T_{\max(\alpha_0,r)+n(r+h^2d)}\leq n)\ \leq\ 4n\,e^{-r^2/8d}\;. \]
\end{lemma}

\begin{remark}[Improved Stability of NUTS]\label{rem:stab_fixed}
The bound of Lemma \ref{lem:stab_leapfrog} is based on the maximum expansion \eqref{eq:LFsubset} of one transition.  Alternatively, switching to the exact Hamiltonian flow for simplicity, the central limit theorem gives for $x\in D_\alpha$
\begin{align*}
    |\Pi(\phi_t(x,v))|^2-|x|^2\ &=\ -\sin^2t\,(|x|^2-|v|^2)+\sin(2t)x\cdot v\\
    &\approx\ -\sin^2t\,(|x|^2-d)+\sqrt{\sin^4t+\sin^2(2t)}\,d^{1/2}Z
\end{align*}
with $Z \sim \mathcal{N}(0,1)$ standard normally distributed.
The evolution of $|X_k|^2$ for the corresponding chain is hence governed by a drift towards $d$ together with additive noise, which sparks hope for stability to hold in a fixed spherical shell.
This could be used to improve the assertion of Theorem \ref{thm:NUTSmixing} by logarithmic factors.
However, this evolution depends on the path length which is non-trivial for NUTS.
Therefore, we are satisfied with Lemma \ref{lem:stab_leapfrog} and leave the study of this slight improvement to future work.
\end{remark}

\section{Mixing of Uniform HMC}\label{sec:UnifHMC}

In the last section, we have seen that NUTS can be interpreted as an accept/reject chain in the sense of Section \ref{sec:ARchains} with $\pi_{\A}$ corresponding to Uniform HMC and the accept event $A(x)=E_{\alpha,r}\cap A_{\mathrm{index}}(x)$.  In order to employ the framework given in Theorem \ref{thm:ARmix}, we require insight into the mixing of Uniform HMC, which we now explore via couplings.

The path length distribution $\mathrm{Law}(hL)$ in physical time of Uniform HMC (as defined on page \pageref{alg:UnifHMC}) is a distribution over the collection  $h \cup_{I \in \mathfrak{I}_0(k)} I$.  This distribution is not  uniform over this collection.  Instead, the probability mass function (PMF)  follows a triangular pattern, given explicitly by
\begin{equation} \label{eq:tau_uniformHMC}
\tau_k(\{t\})\ =\ \frac{2^k-|t/h|}{(2^k)^2}\quad\text{for $t\in h\,\mathbb Z$ such that $|t|\leq h(2^k-1)$.} \end{equation}
The numerator measures the overlap of the index orbits in $\mathfrak I_0(k)$ at the index $t/h\in\mathbb Z$, while the denominator normalizes by the total number of indices across the $2^k$ index orbits, each containing $2^k$ indices.  As a result, the PMF forms a triangular shape with width proportional to the maximal integration time, as illustrated in Figure~\ref{fig:tri}.  With the auxiliary variables $v\sim\gamma$ and $T\sim\tau_k$, a transition $X\sim\pi_{\mathrm{UnifHMC}}(x,\cdot)$ can hence be written as $X=\Pi(\Phi_h^{T/h}(x,v))$.

\begin{figure}[t]
\centering
\begin{tikzpicture}
        \def\k{5} 
        \def\h{0.11} 
        \def\pii{3.14} 
        \def\piii{3.14159} 
     
    \begin{axis}[
        width=12cm, height=6cm, 
        xlabel={$t$ (physical time)},
        ylabel={$\tau_k(\{t\})$ (PMF)},
        grid=none,
        xtick = {-\piii, 0, \piii},
        xticklabels={$-\pi$, $0$, $\pi$},
        xtick distance=1,
        ytick distance=3,
        ymin=0,ymax=0.04,
        xmin=-3.34, xmax=3.34,
        yticklabels={},
        axis x line=bottom,
        axis y line=none,
        cycle list name=color list]

        \addplot [
            thick, 
            gray, 
            domain=-\pii:\pii, 
            samples=2^\k-1, 
            ycomb, 
            mark=*] 
        {max(0, (2^\k - abs(x/\h))/(2^\k)^2)};
        
        \addlegendentry{$\tau_k(\{t\})$  (PMF)}
    \end{axis}
\end{tikzpicture}
\caption{\textit{The PMF of the triangular distribution in \eqref{eq:tau_uniformHMC} with $k=5$ and $h=0.11$, as in Figure~\ref{fig:orbitfig} (c).}} 
\label{fig:tri}
\end{figure}
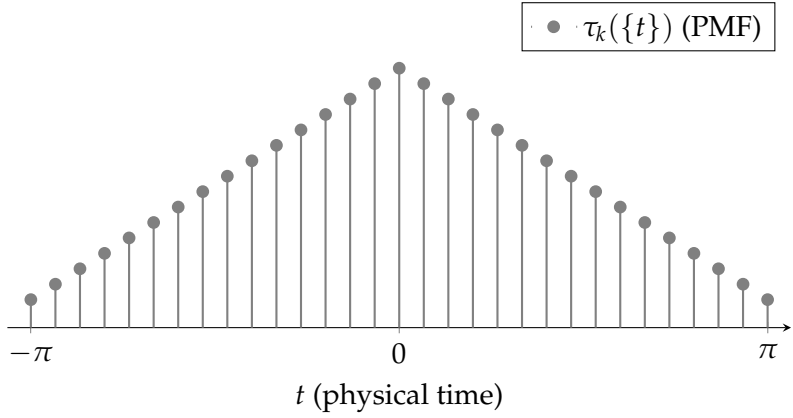

The crux of Uniform HMC for our purpose of analysis via couplings is that neither of the auxiliary variables depend on $x$.  To fully appreciate this significance, let us recapitulate the idea behind the coupling approach: We deduce conclusions from observing two suitably related -- coupled -- copies of the stochastic process.  The coupling specifically relates the randomness involved in the transitions. Since the leapfrog flow is deterministic, in Uniform HMC this corresponds to coupling the auxiliary variables.  Due to the fact that the copies of the process being coupled typically are in distinct states, this is greatly simplified by the auxiliary variables being state independent.

The following two lemmas verify Assumptions (i) and (ii) of Theorem \ref{thm:ARmix} for general locally adaptive HMC kernels of the form \eqref{eq:HMCkernel}.

\medskip

\begin{lemma}\label{lem:contr}
    Let $\pi$ be a kernel of the form \eqref{eq:HMCkernel} with invariant distribution $\gamma$.
	Then, it holds for all $x,\tilde x\in\mathbb R^d$
	\[ \mathcal W^1_{|\cdot|}\bigr(\pi(x,\cdot),\,\pi(\tilde x,\cdot)\bigr)\ \leq\ \int|\cos(\beta_h t)|\tau(dt)\,|x-\tilde x| \;. \]
\end{lemma}

For Uniform HMC with orbit length $h(2^k-1)$ bounded from below by an absolute constant, the integral on the right hand side with $\tau_k$ is bounded above by $1-\rho$ for an absolute constant $\rho>0$, so that
\[ \mathcal W^1_{|\cdot|}\bigr(\pi_{\mathrm{UnifHMC}}(x,\cdot),\,\pi_{\mathrm{UnifHMC}}(\tilde x,\cdot)\bigr)\ \leq\ (1-\rho)\,|x-\tilde x| \;. \]

\begin{proof}
Let $x,\tilde x\in\mathbb R^d$ and consider the synchronous coupling
\[ (X,\widetilde X)\ =\ \bigr(\Pi(\Phi_h^{t/h}(x,v)),\,\Pi(\Phi_h^{T/h}(\tilde x,v))\bigr) \]
of $\pi(x,\cdot)$ and $\pi(\tilde x,\cdot)$ obtained by using the same path length $T\sim\tau$ and velocity $v\sim\gamma$ for both copies.
It then holds
\[ \mathcal W^1_{|\cdot|}\bigr(\pi(x,\cdot),\,\pi(\tilde x,\cdot)\bigr)\ \leq\ \mathbb E|X-\widetilde X|\ \leq\ \mathbb E|\cos(\beta_h T)|\,|x-\tilde x|\ =\ \int|\cos(\beta_h t)|\tau(dt)\,|x-\tilde x|\;, \]
where we inserted the explicit leapfrog flow \eqref{eq:GaussianLF}.
\end{proof}

\medskip

\begin{lemma}\label{lem:OS}
    Let $\pi$ be a kernel of the form \eqref{eq:HMCkernel} with invariant distribution $\gamma$.
	Then, it holds for all $B$ such that $\pi\,\mathbb Z\cap h\,\mathbb Z\subseteq B\subseteq h\,\mathbb Z$ and all $x,\tilde x\in\mathbb R^d$
	\[ \TV\bigr(\pi(x,\cdot),\,\pi(\tilde x,\cdot)\bigr)\ \leq\ \int_{B^c}|\cot(\beta_h t)|\tau(dt)\,|x-\tilde x|\ +\ \tau(B)\;. \]
\end{lemma}

Differentiating between $B$ and $B^c$ is necessary if $\tau(\pi\,\mathbb Z)\ne0$ since $|\cot|=\infty$ in these points.  Generally, $|\cot|$ diverging as the path length approaches $\pi\,\mathbb Z$ reflects the period which limits the effective motion of the coupled copies and therefore their ability to meet exactly.  Depending on $\tau$, $B$ should be chosen to minimize the right hand side.

For Uniform HMC with orbit length $h(2^k-1)$ bounded from below by an absolute constant, one can choose $B$ such that
\[ \TV\bigr(\pi_{\mathrm{UnifHMC}}(x,\cdot),\,\pi_{\mathrm{UnifHMC}}(\tilde x,\cdot)\bigr)\ \leq\ C_{Reg}\,|x-\tilde x|\ +\ c \]
for absolute constants $C_{Reg}>0$ and $0<c<1$, balancing the two terms in the lemma.

\begin{proof}
Let $x,\tilde x\in\mathbb R^d$, $T\sim\tau$ and $v,\tilde v\sim\gamma$.
Consider the coupling
\[ (X,\widetilde X)\ =\ \bigr(\Pi(\Phi_h^{T/h}(x,v)),\,\Pi(\Phi_h^{T/h}(\tilde x,\tilde v))\bigr) \]
of $\pi(x,\cdot)$ and $\pi(\tilde x,\cdot)$ using the same path length but distinct velocities $v$ and $\tilde v$ coupled such that with maximal probability
\[ \tilde v\ =\ v+\cot(\beta_hT)(1-h^2/4)^{1/2}\,(x-\tilde x) \]
in which case the copies meet, i.e., $X=\widetilde X$.
Writing $m(t)=\cot(\beta_ht)(1-h^2/4)^{1/2}\,(x-\tilde x)$, we then have for any $B$ such that $\pi\,\mathbb Z\cap h\,\mathbb Z\subseteq B\subseteq h\,\mathbb Z$
\begin{align*}
	\TV\bigr(\pi(x,\cdot),\,\pi(\tilde x,\cdot)\bigr)
\ &\leq\ \mathbb P(X\ne\widetilde X)
\ \leq\ \int_{B^c}\mathbb P\bigr(\tilde v\ne v+m(t)\bigr)\tau(dt)+\tau(B) \\
&=\ \int_{B^c}\TV\bigr(\gamma+m(t),\gamma\bigr)\tau(dt)+\tau(B)\;,
\end{align*}
where we used maximality of the coupling.
By Pinsker's inequality
\[ \TV\bigr(\gamma+m(t),\gamma\bigr)\ \leq\ \sqrt2\, \mathrm{H}\bigr(\gamma\bigr|\gamma+m(t)\bigr)^{1/2}\ =\ |m(t)| \]
so that
\[ \TV\bigr(\pi(x,\cdot),\,\pi(\tilde x,\cdot)\bigr)\ \leq\ \int_{B^c}|m(t)|\tau(dt)+\tau(B)\ \leq\ \int_{B^c}|\cot(\beta_h t)|\tau(dt)\,|x-\tilde x| + \tau(B)\;. \]
\end{proof}

\section{Mixing of NUTS: Proof of Main Result}\label{sec:NUTSproof}

\begin{proof}[Proof of Theorem \ref{thm:NUTSmixing}]

We apply the framework for the mixing of accept/reject Markov chains in the high acceptance regime, as introduced in Section \ref{sec:ARchains}, to NUTS.
Let the step size $h$ satisfy $0<h\leq1$ and condition \eqref{eq:h_rest}. Further, assume the maximum orbit length, $h(2^{k_{\mathrm{max}}}-1)$, to be bounded from below by an absolute constant.
The state space $S=\mathbb R^d$ is the Euclidean space, and the domain $D$ is defined to be a spherical shell $D_\alpha$ where $\alpha\leq d$, as described in \eqref{eq:D_alpha}.

Let $0\leq r\leq d$, and define $\delta$ as in \eqref{eq:u-turn_delta}.  Lemma \ref{prop:NUTStoUnifHMC} then ensures that, within the shell $D_\alpha$, NUTS fits the definition of an accept/reject chain, with the accept kernel $\pi_{\A}$ corresponding to Uniform HMC with an index orbit length $2^{\min(k_\ast,k_{\mathrm{max}})}$, where $k_\ast$ is given in Lemma \ref{prop:u-turn}. The corresponding acceptance event is given by $A(x)=E_{\alpha,r}\cap A_{\mathrm{index}}(x)$ where $E_{\alpha,r}$ is given in Lemma \ref{lem:E} and $A_{\mathrm{index}}(x)$ is given in Lemma \ref{lem:MultHMCtoUnifHMC}.  The orbit length for this instance of Uniform HMC $h(2^{\min(k_\ast,k_{\mathrm{max}})}-1)$ is bounded from below by an absolute constant due to the assumption on the maximal orbit length of NUTS, and since $h(2^{k_\ast}-1)\geq\pi$ by Lemma \ref{prop:u-turn}. 

Next, we verify Assumptions (i)-(iv) of Theorem \ref{thm:ARmix}.  For simplicity, we treat double logarithmic terms in both the dimension $d$ and the inverse accuracy $\eps^{-1}$ as absolute constants.

Wasserstein contractivity (Assumption (i)) and total variation to Wasserstein regularization (Assumption (ii)) hold for the given instance of Uniform HMC with absolute constants $\rho,C_{Reg}>0$ and $0<c<1$, as established by Lemmas \ref{lem:contr} and \ref{lem:OS}. This is because the orbit length in physical time is bounded from below by an absolute constant.

For Assumption (iii), we require
\begin{equation}\label{eq:requiredBd}
    2\,\mathfrak E\,\sup\nolimits_{x\in D_\alpha}\mathbb{P}\bigr((E_{\alpha,r}\cap A_{\mathrm{index}}(x))^c\bigr)\ +\ C_{Reg}\,\diam_{|\cdot|}(D_\alpha)\,\exp\bigr(-\rho(\mathfrak E-1))\ +\ b\ \leq\ 1\ -\ c
\end{equation}
to hold for some constant $b>0$.
Inserting $\rho,C_{Reg}$ and $c$ together with the bound
\[ \sup\nolimits_{x\in D_\alpha}\mathbb P\bigr((E_{\alpha,r}\cap A_{\mathrm{index}}(x))^c\bigr)\ \leq\ 4\,e^{-r^2/8d}+h^2\max(\alpha,r)+\frac14h^4d \]
as stated in Lemma \ref{prop:NUTStoUnifHMC} together with the fact that $\diam_{\mathsf{d}}(D_\alpha)\leq2\sqrt{2d}$ for $\alpha\leq d$, we can conclude the existence of absolute constants $b>0$ and $c',C',C''>0$ such that \eqref{eq:requiredBd} holds with
\[ \mathfrak E\ =\ C'\log d\;,\quad r\ =\ C''\sqrt d\;,\quad\text{and all}\quad h\ \leq\ \bar h\ =\ c'\max(\alpha,\sqrt d)^{-1/2}\log^{-1/2}d\;. \]

Set the horizon $\mathfrak H=C'''\log d\log\eps^{-1}$ for an absolute constant $C'''>0$. For Assumption (iv), we  require exit probability bounds over this horizon for NUTS initialized  in both the initial distribution $\nu$ and in the stationary distribution $\mu$. By Lemma \ref{lem:stab_leapfrog}, if we take $\alpha=C''''\max\bigr(\alpha_0,\sqrt d\log d\log^{3/2}\eps^{-1}\bigr)$ with some absolute constant $C''''>0$, it holds for all $x\in D_{\alpha_0}$ that $\mathbb P_x(T_\alpha\leq \mathfrak H)\leq\eps/8.$
Additionally, with the assumption that $\max\bigr(\nu(D_{\alpha_0}^c),\gamma(D_{\alpha_0}^c)\bigr)\leq\eps/8$, this implies the desired exit probability bounds, completing the requirements for Assumption (iv).

Thus, all the assumptions of Theorem \ref{thm:ARmix} are satisfied, thereby establishing the required upper bound on the mixing time of NUTS.
\end{proof}

\printbibliography

@misc{monmarche2022hmc,
  title={HMC and Langevin united in the unadjusted and convex case},
  author={Monmarch{\'e}, Pierre},
      archivePrefix={arXiv},
  eprint={2202.00977},
  year={2022},
    primaryClass={math.PR}
}

@misc{betancourt2013generalizing,
  title={Generalizing the no-U-turn sampler to Riemannian manifolds},
  author={Betancourt, Michael J},
  primaryClass={stat.ME},
  archivePrefix={arXiv},
    eprint={1304.1920},
  year={2013}
}

@misc{camrud2023second,
  title={Second order quantitative bounds for unadjusted generalized Hamiltonian Monte Carlo},
  author={Camrud, Evan and Durmus, Alain and Monmarch{\'e}, Pierre and Stoltz, Gabriel},
  archivePrefix={arXiv},  
primaryClass={math.PR},
eprint={2306.09513},
  year={2023}
}

@article{monmarche2024entropic,
  title={An entropic approach for {Hamiltonian Monte Carlo}: the idealized case},
  author={Monmarch{\'e}, Pierre},
  journal={Annals of Applied Probability},
  volume={34},
  number={2},
  pages={2243--2293},
  year={2024}
}

@article{monmarche2021high,
  title={High-dimensional MCMC with a standard splitting scheme for the underdamped Langevin diffusion.},
  author={Monmarch{\'e}, Pierre},
  journal={Electronic Journal of Statistics},
  volume={15},
  number={2},
  pages={4117--4166},
  year={2021}
}

@article{article,
author = {Andersen, Hans and Diaconis, Persi},
year = {2007},
month = {01},
pages = {},
title = {Hit and Run as a Unifying Device},
volume = {148},
journal = {J. Soc. Fr. Stat. & Rev. Stat. Appl.}
}

@misc{andrieu2020general,
  title={A general perspective on the Metropolis-Hastings kernel},
  author={Andrieu, Christophe and Lee, Anthony and Livingstone, Sam},
  archivePrefix={arXiv},
  primaryClass={stat.CO},
  eprint={2012.14881},
  year={2020}
}

@misc{BouRabeeCarpenterMarsden2024,
  title={GIST: Gibbs self-tuning for locally adaptive Hamiltonian Monte Carlo},
  author={Bou-Rabee, Nawaf and Carpenter, Bob and Marsden, Milo},
  archivePrefix={arXiv},
  primaryClass={stat.CO},
    eprint={2404.15253},
  year={2024}
}

@inproceedings{erdogdu2021convergence,
  title={On the convergence of Langevin Monte Carlo: The interplay between tail growth and smoothness},
  author={Erdogdu, Murat A and Hosseinzadeh, Rasa},
  booktitle={Conference on Learning Theory},
  pages={1776--1822},
  year={2021},
  organization={PMLR}
}

@article{dalalyan2017theoretical,
  title={Theoretical guarantees for approximate sampling from smooth and log-concave densities},
  author={Dalalyan, Arnak S},
  journal={Journal of the Royal Statistical Society Series B: Statistical Methodology},
  volume={79},
  number={3},
  pages={651--676},
  year={2017},
  publisher={Oxford University Press}
}

@article{durmus2017nonasymptotic,
  title={Nonasymptotic convergence analysis for the unadjusted Langevin algorithm},
  author={Durmus, Alain and Moulines, Eric},
  journal={Annals of Applied Probability},
  volume={27},
  number={3},
  pages={1551--1587},
  year={2017}
}

@misc{durmus2023convergence,
  title={On the convergence of dynamic implementations of {Hamiltonian Monte Carlo and no U-turn samplers}},
  author={Durmus, Alain and Gruffaz, Samuel and Kailas, Miika and Saksman, Eero and Vihola, Matti},
  archivePrefix={arXiv},
  primaryClass={stat.CO},
  eprint={2307.03460},
  year={2023}
}

@article{carpenter2016stan,
	Author = {Carpenter, B. and Gelman, A. and Hoffman, M. and Lee, D. and Goodrich, B. and Betancourt, M. and Brubaker, M. A. and Guo, J. and Li, P. and Riddell, A.},
	Journal = {Journal of Statistical Software},
	Pages = {1--37},
	Title = {Stan: A probabilistic programming language},
	Volume = {20},
	Year = {2016}}

@article{EbMa2019,
author = {Andreas Eberle and Mateusz B. Majka},
title = {{Quantitative contraction rates for Markov chains on general state spaces}},
volume = {24},
journal = {Electronic Journal of Probability},
number = {none},
publisher = {Institute of Mathematical Statistics and Bernoulli Society},
pages = {1 -- 36},
year = {2019}
}

@article{BouRabeeSchuh2023,
author = {Bou-Rabee, Nawaf and Schuh, Katharina},
title = {{Convergence of unadjusted Hamiltonian Monte Carlo for mean-field models}},
volume = {28},
journal = {Electronic Journal of Probability},
publisher = {Institute of Mathematical Statistics and Bernoulli Society},
pages = {1 -- 40},
year = {2023}
}

@article{BoEb2022,
	Author = {Bou-Rabee, N. and Eberle, A.},
	Journal = {Ann. Inst. H. Poincar{\'e} Probab. Statist},
	Number = {2},
	Pages = {916-944},
	Title = {Couplings for {A}ndersen dynamics in high dimension},
	Volume = {58},
	Year = {2022}}

@article{BouRabeeEberle2023,
	Author = {Bou-Rabee, Nawaf and Eberle, Andreas},
	Journal = {Bernoulli},
	Number = {1},
	Pages = {75-104},
	Title = {Mixing Time Guarantees for Unadjusted Hamiltonian Monte Carlo},
	Volume = {29},
	Year = {2023}}

@article{chewi2021optimal,
	Author = {Chewi, Sinho and Lu, Chen and Ahn, Kwangjun and Cheng, Xiang and Gouic, Thibaut Le and Rigollet, Philippe},
	Journal = {Conference on Learning Theory},
    pages = {1260–1300},
	Title = {Optimal dimension dependence of the Metropolis-Adjusted Langevin Algorithm},
	Year = {2021},
    organization={PMLR}
}

@article{LST2020,
	Author = {Lee, Y. T. and Shen, R. and Tian, K.},
	Journal = {Conference on Learning Theory},
    pages = {2565–2597},
	Title = {Logsmooth gradient concentration and tighter runtimes for Metropolized Hamiltonian Monte Carlo},
	Year = {2020},
}

@article{durmus2019high,
  title={High-dimensional Bayesian inference via the unadjusted Langevin algorithm},
  author={Durmus, Alain and Moulines, Eric},
  journal={Bernoulli},
  volume={25},
  number={4A},
  pages={2854--2882},
  year={2019},
  publisher={Bernoulli Society for Mathematical Statistics and Probability}
}

@article{dalalyan2020sampling,
  title={On sampling from a log-concave density using kinetic Langevin diffusions},
  author={Dalalyan, Arnak S and Riou-Durand, Lionel},
  journal={Bernoulli},
  volume={26},
  number={3},
  pages={1956--1988},
  year={2020},
  publisher={Bernoulli Society for Mathematical Statistics and Probability}
}

@inproceedings{cheng2018underdamped,
	Author = {Cheng, X. and Chatterji, N. S. and Bartlett, P. L. and Jordan, M. I.},
	Booktitle = {Conference On Learning Theory},
	Pages = {300--323},
	Title = {{Underdamped Langevin MCMC: A non-asymptotic analysis}},
	Year = {2018}}

@article{deligiannidis2018randomized,
	Author = {Deligiannidis, G. and Paulin, D. and Bouchard-C{\^o}t{\'e}, A. and Doucet, A.},
	Journal = {arXiv preprint},
    number = {1808.04299},
	Title = {Randomized {Hamiltonian Monte Carlo} as scaling limit of the bouncy particle sampler and dimension-free convergence rates},
	Year = {2018}}

@article{roberts2002one,
	Author = {Roberts, G. and Rosenthal, J.},
	Journal = {Stochastic processes and their applications},
	Number = {2},
	Pages = {195--208},
	Publisher = {Elsevier},
	Title = {One-shot coupling for certain stochastic recursive sequences},
	Volume = {99},
	Year = {2002}}

@article{BoEbZi2020,
	Author = {Bou-Rabee, Nawaf and Eberle, Andreas and Zimmer, Raphael},
	Fjournal = {Annals of Applied Probability},
	Journal = {Annals of Applied Probability},
	Number = {3},
	Pages = {1209--1250},
	Publisher = {The Institute of Mathematical Statistics},
	Title = {Coupling and convergence for Hamiltonian Monte Carlo},
	Volume = {30},
	Year = {2020}}

@article{BouRabeeOberdoerster2024,
	author = {Nawaf Bou-Rabee and Stefan Oberd{\"o}rster},
	journal = {Electronic Journal of Probability},
	pages = {1 -- 27},
	title = {{Mixing of Metropolis-adjusted Markov chains via couplings: The high acceptance regime}},
	volume = {29},
	year = {2024}}

@article{DiHoNe2000,
	Author = {Diaconis, P. and Holmes, S. and Neal, R. M.},
	Journal = {Annals of Applied Probability},
	Pages = {726--752},
	Publisher = {JSTOR},
	Title = {Analysis of a Nonreversible {M}arkov Chain Sampler},
	Year = {2000}}

@article{madras2010quantitative,
	Author = {Madras, Neal and Sezer, Deniz},
	Journal = {Bernoulli},
	Number = {3},
	Pages = {882--908},
	Publisher = {Bernoulli Society for Mathematical Statistics and Probability},
	Title = {Quantitative bounds for Markov chain convergence: Wasserstein and total variation distances},
	Volume = {16},
	Year = {2010}}

@article{BoSa2017,
	Author = {Bou-Rabee, Nawaf and Sanz-Serna, Jes{\'u}s Mar{\'\i}a},
	Fjournal = {Annals of Applied Probability},
	Journal = {Annals of Applied Probability},
	Number = {4},
	Pages = {2159--2194},
	Publisher = {The Institute of Mathematical Statistics},
	Title = {Randomized Hamiltonian Monte Carlo},
	Volume = {27},
	Year = {2017}}

@article{Ma1989,
	Author = {Mackenzie, P. B.},
	Journal = {Physics Letters B},
	Number = {3},
	Pages = {369--371},
	Publisher = {Elsevier},
	Title = {An improved hybrid {M}onte {C}arlo method},
	Volume = {226},
	Year = {1989}}

@article{HoGe2014,
	Author = {Hoffman, M. D. and Gelman, A.},
	Journal = {Journal of Machine Learning Research},
	Number = {1},
	Pages = {1593--1623},
	Publisher = {JMLR. org},
	Title = {The no-U-turn sampler: Adaptively setting path lengths in Hamiltonian Monte Carlo},
	Volume = {15},
	Year = {2014}}

@article{Ne2011,
	Author = {Neal, R. M.},
	Journal = {Handbook of {M}arkov {C}hain {M}onte {C}arlo},
	Pages = {113-162},
	Title = {{MCMC} using {H}amiltonian dynamics},
	Volume = {2},
	Year = {2011}}

@misc{betancourt2013hamiltonianmontecarlohierarchical,
      title={Hamiltonian Monte Carlo for Hierarchical Models}, 
      author={M. J. Betancourt and Mark Girolami},
      year={2013},
      eprint={1312.0906},
      archivePrefix={arXiv},
      primaryClass={stat.ME}
}

@misc{DurmusEberle2021,
	Author = {Durmus, A. and Eberle, A.},
      archivePrefix={arXiv},
      eprint={2108.00682},
      year={2021},
      primaryClass={math.PR},
	Title = {Asymptotic bias of inexact Markov Chain Monte Carlo Methods in High Dimension}
}

@article{kleppe2022connecting,
  title={Connecting the dots: Numerical randomized {H}amiltonian {M}onte {C}arlo with state-dependent event rates},
  author={Kleppe, Tore Selland},
  journal={Journal of Computational and Graphical Statistics},
  volume={31},
  number={4},
  pages={1238--1253},
  year={2022},
  publisher={Taylor \& Francis}
}

@article{BePiRoSaSt2013,
	Author = {Beskos, A. and Pillai, N. S. and Roberts, G. O. and Sanz-Serna, J. M. and Stuart, A. M.},
	Journal = {Bernoulli},
	Pages = {1501-1534},
	Title = {Optimal tuning of the Hybrid {M}onte {C}arlo algorithm},
	Volume = {19},
	Year = {2013}}

@article{DuKePeRo1987,
	Author = {Duane, S. and Kennedy, A. D. and Pendleton, B. J. and Roweth, D.},
	Journal = {Physics Letters B},
	Pages = {216--222},
	Title = {Hybrid {M}onte-{C}arlo},
	Volume = {195},
	Year = {1987}}

@misc{betancourt2017conceptual,
  title={A conceptual introduction to Hamiltonian Monte Carlo},
  author={Betancourt, Michael},
  archivePrefix={arXiv},
  primaryClass={stat.ME},
  eprint={1701.02434},
  year={2017}
}

@article{GiCa2011,
	Author = {Girolami, M. and Calderhead, B.},
	Journal = {J R Statist Soc B},
	Pages = {123-214},
	Title = {Riemann manifold {L}angevin and {H}amiltonian {M}onte {C}arlo methods},
	Volume = {73},
	Year = {2011}}

@article{salvatier2016probabilistic,
  title={Probabilistic programming in Python using PyMC3},
  author={Salvatier, John and Wiecki, Thomas V and Fonnesbeck, Christopher},
  journal={PeerJ Computer Science},
  volume={2},
  pages={e55},
  year={2016},
  publisher={PeerJ Inc.}
}

@inproceedings{ge2018t,
  author    = {Hong Ge and
               Kai Xu and
               Zoubin Ghahramani},
  title     = {Turing: A language for flexible probabilistic inference},
  booktitle = {International Conference on Artificial Intelligence and Statistics, ({AISTATS})},
  pages     = {1682--1690},
  year      = {2018}
}

@article{nimble-article:2017, 
  author = {{de Valpine}, P. and Turek, D. and Paciorek, C.J. and Anderson-Bergman, C. and {Temple Lang}, D. and Bodik, R.},
  title = {Programming with models: writing statistical algorithms for general model structures with {NIMBLE}},
  year = {2017}, 
  journal = {Journal of Computational and Graphical Statistics},
  volume = 26,
  pages = {403-417}
}

@misc{phan2019composable,
  title={Composable effects for flexible and accelerated probabilistic programming in {N}um{P}yro},
  author={Phan, Du and Pradhan, Neeraj and Jankowiak, Martin},
  archivePrefix={arXiv},
  eprint={1912.11554},
  primaryClass={stat.ML},
  year={2019}
}

@article{SherlockUrbasLudkin2023Apogee,
author = {Sherlock, Chris and Urbas, Szymon and Ludkin, Matthew},
title = {The apogee to apogee path sampler},
journal = {Journal of Computational and Graphical Statistics},
volume = {32},
number = {4},
pages = {1436--1446},
year = {2023},
publisher = {Taylor \& Francis}
}

@book{Vershynin,
    AUTHOR = {Vershynin, Roman},
     TITLE = {High-dimensional probability},
    SERIES = {Cambridge Series in Statistical and Probabilistic Mathematics},
    VOLUME = {47},
 PUBLISHER = {Cambridge University Press, Cambridge},
      YEAR = {2018}
}

@article{Talagrand_NewLook,
    AUTHOR = {Talagrand, Michel},
     TITLE = {A new look at independence},
   JOURNAL = {Ann. Probab.},
  FJOURNAL = {The Annals of Probability},
    VOLUME = {24},
      YEAR = {1996},
    NUMBER = {1},
     PAGES = {1--34}
}

@article{Chen_Minimax,
    AUTHOR = {Wu, K. and Schmidler, S. and Chen, Y.},
     TITLE = {Minimax mixing time of the metropolis-adjusted {L}angevin
              algorithm for log-concave sampling},
   JOURNAL = {J. Mach. Learn. Res.},
  FJOURNAL = {Journal of Machine Learning Research (JMLR)},
    VOLUME = {23},
      YEAR = {2022},
     PAGES = {Paper No. [270], 63}
}

@misc{chen2023,
      title={When does Metropolized Hamiltonian Monte Carlo provably outperform Metropolis-adjusted Langevin algorithm?}, 
      author={Chen, Yuansi and Gatmiry, Khashayar},
      year={2023},
      primaryClass={stat.CO},
      eprint={2304.04724},
      archivePrefix={arXiv}
}

@book {TalagrandLedoux_ProbabilityIn,
    AUTHOR = {Ledoux, Michel and Talagrand, Michel},
     TITLE = {Probability in {B}anach spaces},
    SERIES = {Classics in Mathematics},
      NOTE = {Isoperimetry and processes,
              Reprint of the 1991 edition},
 PUBLISHER = {Springer-Verlag, Berlin},
      YEAR = {2011},
     PAGES = {xii+480}
}

@misc{apers2022,
      title={Hamiltonian Monte Carlo for efficient Gaussian sampling: long and random steps}, 
      author={Simon Apers and Sander Gribling and Dániel Szilágyi},
      year={2022},
    primaryClass={stat.ML},
      eprint={2209.12771},
      archivePrefix={arXiv}
}

@article {LS1993,
    AUTHOR = {Lov\'asz, L. and Simonovits, M.},
     TITLE = {Random walks in a convex body and an improved volume
              algorithm},
   JOURNAL = {Random Structures Algorithms},
    VOLUME = {4},
      YEAR = {1993},
    NUMBER = {4},
     PAGES = {359--412}
}

@article {Lovasz99,
    AUTHOR = {Lov\'asz, L.},
     TITLE = {Hit-and-run mixes fast},
   JOURNAL = {Math. Program.},
    VOLUME = {86},
      YEAR = {1999},
    NUMBER = {3},
     PAGES = {443--461}
}

@inproceedings{chen1999lifting,
  title={Lifting Markov chains to speed up mixing},
  author={Chen, Fang and Lov{\'a}sz, L{\'a}szl{\'o} and Pak, Igor},
  booktitle={Proceedings of the thirty-first annual ACM symposium on Theory of computing},
  pages={275--281},
  year={1999}
}

@article {CaoLuWang23,
    AUTHOR = {Cao, Yu and Lu, Jianfeng and Wang, Lihan},
     TITLE = {On explicit {$L^2$}-convergence rate estimate for underdamped
              {L}angevin dynamics},
   JOURNAL = {Arch. Ration. Mech. Anal.},
    VOLUME = {247},
      YEAR = {2023},
    NUMBER = {5},
     PAGES = {Paper No. 90, 34}
}

@article {EberleLoerler24,
    AUTHOR = {Eberle, Andreas and L\"orler, Francis},
     TITLE = {Non-reversible lifts of reversible diffusion processes and relaxation times},
   JOURNAL = {Probab. Theory Relat. Fields},
      YEAR = {2024}
}

@article {Eberle14,
    AUTHOR = {Eberle, Andreas},
     TITLE = {Error bounds for {M}etropolis-{H}astings algorithms applied to
              perturbations of {G}aussian measures in high dimensions},
   JOURNAL = {Annals of Applied Probability},
    VOLUME = {24},
      YEAR = {2014},
    NUMBER = {1},
     PAGES = {337--377}
}

\end{document}